\documentclass[a4paper,11pt]{article}

\usepackage{ucs}
\usepackage[utf8x]{inputenc}
\usepackage{type1ec}

\usepackage{hyperref}
\usepackage{fullpage}

\usepackage{bbm}
\usepackage[bbgreekl]{mathbbol}
\usepackage{nccrules}
\usepackage{tocbibind}
\usepackage[intlimits,leqno]{amsmath}
\usepackage{amsthm}
\usepackage{amssymb}
\usepackage[mathscr]{eucal}
\usepackage{stmaryrd}
\usepackage{xspace}
\usepackage[all]{xy}
\newdir{ >}{{}*!/-10pt/\dir{>}}
\newdir{|>}{!/4.5pt/@{|}*:(1,-.2)@^{>}*:(1,+.2)@_{>}}
\DeclareMathOperator*{\bigboxtimes}{\boxtimes}

\newcommand{\Z}{\ensuremath{\mathbb{Z}}}

\newcommand{\R}{\ensuremath{\mathbb{R}}}
\newcommand{\C}{\ensuremath{\mathbb{C}}}

\newcommand{\Aut}{\ensuremath{\mathrm{Aut}}}
\newcommand{\Tr}{\ensuremath{\mathrm{tr}\,}}
\newcommand{\Nrd}{\ensuremath{\mathrm{Nrd}\,}}

\newcommand{\dd}{\ensuremath{\,\mathrm{d}}}

\newcommand{\angles}[1]{\ensuremath{\langle #1 \rangle}}

\newcommand{\identity}{\ensuremath{\mathrm{id}}}

\newcommand{\Hom}{\ensuremath{\mathrm{Hom}}}
\newcommand{\Isom}{\ensuremath{\mathrm{Isom}}}
\newcommand{\End}{\ensuremath{\mathrm{End}}}
\newcommand{\rightiso}{\ensuremath{\stackrel{\sim}{\rightarrow}}}

\newcommand{\Ker}{\ensuremath{\mathrm{Ker}\,}}

\newcommand{\Ad}{\ensuremath{\mathrm{Ad}\,}}

\newcommand{\Gm}{\ensuremath{\mathbb{G}_\mathrm{m}}}

\newcommand{\GL}{\ensuremath{\mathrm{GL}}}

\newcommand{\SU}{\ensuremath{\mathrm{SU}}}
\newcommand{\PGL}{\ensuremath{\mathrm{PGL}}}

\newcommand{\SL}{\ensuremath{\mathrm{SL}}}

\newcommand{\Ind}{\ensuremath{\mathrm{Ind}}}

\newcommand{\Lgrp}[1]{\ensuremath{{}^{\mathrm{L}} #1}}
\newcommand{\WD}{\ensuremath{\mathrm{WD}}}

\theoremstyle{plain}
\newtheorem{proposition}{Proposition}[subsection]
\newtheorem{lemma}[proposition]{Lemma}
\newtheorem{theorem}[proposition]{Theorem}
\newtheorem{corollary}[proposition]{Corollary}

\theoremstyle{definition}
\newtheorem{definition}[proposition]{Definition}
\newtheorem{definition-theorem}[proposition]{Definition-Theorem}
\newtheorem{definition-proposition}[proposition]{Definition-Proposition}
\newtheorem{hypothesis}[proposition]{Hypothesis}
\newtheorem{example}[proposition]{Example}

\newtheorem{remark}[proposition]{Remark}

\renewcommand{\Re}{\ensuremath{\mathrm{Re}\xspace}}
\renewcommand{\Im}{\ensuremath{\mathrm{Im}\xspace}}

\title{Dual $R$-groups of the inner forms of $\SL(N)$}
\author{Kuok Fai Chao \and Wen-Wei Li}
\date{}

\begin{document}

\maketitle

\begin{abstract}
  We study the Knapp-Stein $R$-groups of the inner forms of $\SL(N)$ over a non-archimedean local field of characteristic zero, by using restriction from the inner forms of $\GL(N)$. As conjectured by Arthur, these $R$-groups are then shown to be naturally isomorphic to their dual avatars defined in terms of L-parameters. The $2$-cocycles attached to $R$-groups can be described as well. The proofs are based on the results of K.\! Hiraga and H.\! Saito. We also construct examples to illustrate some new phenomena which do not occur in the case of $\SL(N)$ or classical groups.
\end{abstract}

\tableofcontents

\section{Introduction}
Let $G$ be a connected reductive group over a local field $F$ and $G(F)$ be the locally compact group of the $F$-points of $G$. The study of the tempered representations of $G(F)$ is a crucial ingredient of the monumental work of Harish-Chandra on his Plancherel formula. Denote by $\Pi_{\text{temp}}(G)$ the set of isomorphism classes of irreducible tempered representations, and by $\Pi_{2,\text{temp}}(G)$ its subset of representations which are square-integrable modulo the center. Roughly speaking, elements in $\Pi_{\text{temp}}(G)$ can be obtained as subrepresentations of $I^G_P(\sigma)$, where $P=MU$ is a parabolic subgroup, $\sigma \in \Pi_{2,\text{temp}}(M)$ and $I^G_P(\sigma)$ is the normalized parabolic induction. Assuming the knowledge of square-integrable representations, the study of $\Pi_{\text{temp}}(G)$ then boils down to that of the decomposition of $I^G_P(\sigma)$, for $P$ and $\sigma$ as above.

Knapp, Stein and Silberger (for the non-archimedean case) described the decomposition of $I^G_P(\sigma)$ in terms of the \textit{Knapp-Stein $R$-group} $R_\sigma$. More precisely, we have a central extension of groups
$$ 1 \to \C^\times \to \tilde{R}_\sigma \to R_\sigma \to 1 $$
defined using the normalized intertwining operators $R_P(w, \sigma)$. It is the set $\Pi_-(\tilde{R}_\sigma)$ of the irreducible representations of $\tilde{R}_\sigma$ by which $\C^\times$ acts by $z \mapsto z \cdot\identity$ which governs the decomposition of $I^G_P(\sigma)$. Equivalently, we are given a cohomology class $\mathbf{c}_\sigma \in H^2(R_\sigma, \C^\times)$ attached to this central extension. The group $R_\sigma$ itself suffices to determine whether $I^G_P(\sigma)$ is reducible or not. To extract further information, such as the description of elliptic tempered representations, some knowledge about $\tilde{R}_\sigma$ is also needed. We refer the reader to \cite[\S 2]{Ar93} for details.

On the other hand, the tempered part of the local Langlands correspondence predicts a map $\phi \mapsto \Pi_\phi$ which assigns a finite subset $\Pi_\phi$ of $\Pi_{\text{temp}}(G)$ to every bounded L-parameter $\phi \in \Phi_{\text{bdd}}(G)$, taken up to equivalence, such that
$$ \Pi_{\text{temp}}(G) = \bigsqcup_{\phi \in \Phi_{\text{bdd}}(G)} \Pi_\phi. $$
The internal structure of the tempered L-packets $\Pi_\phi$ is conjectured to be controlled by the $S$-group $S_\phi := Z_{\hat{G}}(\Im(\phi))$. More precisely, following \cite{Ar06}, one has to introduce a central extension
$$ 1 \to \tilde{Z}_\phi \to \tilde{\mathscr{S}}_\phi \to \mathscr{S}_\phi \to 1$$
of finite groups defined in terms of $S_\phi$. The L-packet $\Pi_\phi$ should be in bijection with a set $\Pi(\tilde{\mathscr{S}}_\phi, \chi_G)$ of representations of $\tilde{\mathscr{S}}_\phi$ where $\chi_G$ is a character of $\tilde{Z}_\phi$ coming from Galois cohomology. The relevant definitions will be reviewed later in this article.

The tempered local Langlands correspondence is expected to behave well under normalized parabolic induction, namely for $P=MU$ as above and $\phi_M \in \Phi_{\text{bdd}}(M)$, we deduce $\phi \in \Phi_{\text{bdd}}(G)$ by composing $\phi_M$ with the inclusion $\Lgrp{M} \to \Lgrp{G}$ of L-groups, which is well-defined up to conjugacy. Then $\Pi_\phi$ should be the union of the irreducible constituents of $I^G_P(\sigma)$, where $\sigma$ ranges over the elements of $\Pi_{\phi_M}$. A natural question arises: Is it possible to describe $R_\sigma$, or even $\tilde{R}_\sigma$, in terms of the $S$-groups?

For archimedean $F$ this has been answered by Shelstad \cite{Sh82}; in that case, the extension $\tilde{R}_\sigma \to R_\sigma$ splits and $R_\sigma$ is abelian of exponent two. For general $F$ of characteristic zero, Arthur proposed a generalization in \cite[\S 7]{Ar89-unip} as follows. For every $\phi \in \Phi_{\text{bdd}}(G)$ coming from $\phi_M \in \Phi_{2,\text{bdd}}(M)$ (i.e.\! a parameter for $M$ which is square-integrable modulo the center), he introduced the \textit{dual $R$-group} (also known as the \textit{endoscopic $R$-group}) $R_\phi \simeq \mathscr{S}_{\phi}/\mathscr{S}_{\phi_M}$ and a subgroup $R_{\phi, \sigma} \subset R_\phi$ for every $\sigma \in \Pi_{\phi_M}$. Arthur conjectures a natural isomorphism
$$ R_{\phi, \sigma} \simeq R_\sigma. $$

This has been verified for quasisplit classical groups and unitary groups by Arthur \cite{ArEndo} and Mok \cite{Mok1}, respectively. In their construction of L-packets, the dual $R$-groups play a pivotal role through the \textit{local intertwining relations} (see \cite[Chapter 2]{ArEndo}). It turns out that in these cases, we have $R_\phi = R_{\phi,\sigma}$ and $\tilde{R}_\sigma \to R_\sigma$ splits (see \cite[\S 6.5]{ArEndo}). Similar results are obtained independently by Ban, Goldberg and Zhang \cite{BZ05,Go11,BG12} for non-archimedean $F$. For quaternionic unitary groups, see \cite{Ha04}.

We shall assume hereafter that $F$ is a non-archimedean local field of characteristic zero.
 
Another good test ground for Arthur's conjectures is the group $\SL(N)$ and its inner forms. Indeed, the case $N=2$ is the genesis of endoscopy \cite{LL79}; for general $N$, the local Langlands correspondence for the inner forms $G^\sharp$ of $\SL(N)$ is established in \cite{HS12}, at least in the tempered case. This is based on the local Langlands correspondence for the inner forms $G$ of $\GL(N)$, which satisfies the following nice properties:
\begin{itemize}
  \item the L-packets $\Pi_\phi$ for $G$ are all singletons;
  \item for any parabolic subgroup $P=MU$ and $\sigma \in \Pi_{\text{temp}}(M)$, the induced representation $I^G_P(\sigma)$ is irreducible.
\end{itemize}
In fact, the latter property holds for all unitary $\sigma$, known as Tadić's property (U0) \cite{Sec09}.

The (tempered) local Langlands correspondence for $G^\sharp$ can be obtained by restriction from $G(F)$ to $G^\sharp(F)$; the procedure is somehow dual to the natural projection of L-groups
$$ \mathbf{pr}: \Lgrp{G} \twoheadrightarrow \Lgrp{G^\sharp}. $$

The same recipe can be applied to any Levi subgroup $M$, with respect to $M^\sharp := M \cap G^\sharp$.

The method of restriction provides a convenient device, but we still have to study the internal structure of L-packets for $G^\sharp$ and their behaviour under normalized parabolic induction. For the quasisplit case $G^\sharp = \SL(N)$, such issues can be addressed by the multiplicity-one property of Whittaker models. In that case, the Knapp-Stein $R$-groups are studied in depth in \cite{GK81,GK82,Sh83,Ta92,Go94}. Roughly speaking, let $\sigma^\sharp \in \Pi_{2,\text{temp}}(M^\sharp)$ which lies in the L-packet $\Pi_{\phi^\sharp}$. We may choose $\sigma \in \Pi_{2,\text{temp}}(M)$ such that $\sigma^\sharp \hookrightarrow \sigma|_{M^\sharp}$. Set $\pi := I^G_P(\sigma)$, which is irreducible, then $R_{\sigma^\sharp}$ is described in terms of
$$ X^G(\pi) = \left\{ \eta \in (G(F)/G^\sharp(F))^D : \eta \otimes \pi \simeq \pi \right\} $$
and its analogue $X^M(\sigma)$ for the Levi subgroup $M$ with respect to $M^\sharp$, where $(G(F)/G^\sharp(F))^D$ means the group of continuous characters of $G(F)/G^\sharp(F)$. It is then easy to relate $R_{\sigma^\sharp}$ with $R_{\phi^\sharp}$, and we deduce a canonical isomorphism $R_{\phi^\sharp} = R_{\phi^\sharp, \sigma^\sharp} \simeq R_{\sigma^\sharp}$ as well as a splitting for $\tilde{R}_{\sigma^\sharp} \twoheadrightarrow R_{\sigma^\sharp}$. Note that we used the notations $\phi^\sharp$, $\sigma^\sharp$, etc. to denote the objects attached to $G^\sharp$ and its Levi subgroups.

Whittaker models are no longer available for the non-quasisplit inner forms $G^\sharp$ of $\SL(N)$. What saves the day is the work of Hiraga and Saito \cite{HS12}. They defined a central extension
$$ 1 \to \C^\times \to S^G(\pi) \to X^G(\pi) \to 1 $$
and related it to the central extension of $\mathscr{S}$-groups alluded above. This allows us to study the internal structure of the L-packets obtained by restriction. In the main Theorem \ref{prop:main}, we will prove, among others, that there are
\begin{enumerate}
  \item a canonical isomorphism $R_{\phi^\sharp, \sigma^\sharp} \simeq R_{\sigma^\sharp}$ as conjectured by Arthur;
  \item a ``concrete'' description of the dual $R$-groups for $G^\sharp$, namely
    \begin{align*}
      R_{\phi^\sharp} & \simeq X^G(\pi)/X^M(\sigma), \\
      R_{\phi^\sharp, \sigma^\sharp} & \simeq Z^M(\sigma)^\perp/X^M(\sigma);
    \end{align*}
  \item a description of the class $\mathbf{c}_{\sigma^\sharp} \in H^2(R_{\sigma^\sharp}, \C^\times)$ attached to $\tilde{R}_{\sigma^\sharp} \twoheadrightarrow R_{\sigma^\sharp}$, in terms of the obstruction for extending the representation $\rho$ of $\tilde{\mathscr{S}}_{\phi^\sharp_M}$ to the preimage of $R_{\phi, \sigma^\sharp}$ in $\tilde{\mathscr{S}}_{\phi^\sharp}$.
\end{enumerate}
We refer the reader to \S\ref{sec:res-2} for unexplained notations. Note that the description of $\mathbf{c}_{\sigma^\sharp}$ is also conjectured by Arthur; see \cite[p.537]{Ar96}, \cite[\S 3]{Ar08} or \cite[\S 2.4]{ArEndo} for further discussions.

Arthur's conjecture on $R$-groups for the inner forms of $\SL(N)$ is thus verified. The examples are probably more interesting, however. In \S\ref{sec:examples} we will give conceptual constructions of $\phi^\sharp$ and $\sigma^\sharp$ as above such that
\begin{enumerate}
  \item $\tilde{R}_{\sigma^\sharp} \twoheadrightarrow R_{\sigma^\sharp}$ is not split, or
  \item $R_{\phi^\sharp, \sigma^\sharp} \subsetneq R_{\phi^\sharp}$.
\end{enumerate}

Such phenomena do not occur to the quasisplit classical groups, the quaternionic unitary groups, or $\SL(N)$. The first example is perhaps more surprising, since $\tilde{R}_{\sigma^\sharp} \twoheadrightarrow R_{\sigma^\sharp}$ always splits for generic inducing data. Keys \cite[\S 6]{Ke87} constructed a Knapp-Stein $R$-group with non-split cocycle in the non-connected setting; our example seems to be the only known case for connected reductive groups. In both examples, the relation between $R_{\sigma^\sharp}$ and the $S$-groups is crucial.

In view of the possible applications to automorphic representations, one should also consider certain nontempered unitary representations, namely those appearing in the \textit{A-packets}; see Remark \ref{rem:nontempered} for a short discussion.

Shortly after the release of the first version of our preprint, we were informed of the independent work \cite{CG12} by Choiy and Goldberg that treats the same problems except that of cocycles. Despite some overlaps, their work has a completely different technical core, namely the transfer of Plancherel measures between inner forms, which should have wide-ranging applications.

\paragraph{Organization of this article}
In \S\ref{sec:preliminaries}, we recapitulate the formalism of normalized intertwining operators and Knapp-Stein $R$-groups. We follow Arthur's notations in \cite{Ar89-IOR1,Ar93} closely. In particular, the $R$-group $R_\sigma$ is regarded as a quotient of the isotropy group $W_\sigma$, instead of a subgroup.

In \S\ref{sec:res}, we set up a general formalism of restriction of representations. These results are scattered in \cite{Sh83,Ke87,Ta92,HS12}, just to mention a few. In view of the possible sequels of this work, the behaviour under restriction of normalized intertwining operators is treated in generality.

A special assumption is made in \S\ref{sec:res-2} (Hypothesis \ref{hyp:irred}), namely the parabolically induced representations in question should be irreducible. We are then able to deduce finer information on $R$-groups and their cocycles in this setting. The arguments are not too difficult, but require some careful manipulations.

In \S\ref{sec:SL}, we will specialize to the inner forms of $\SL(N)$ and reformulate the results of Hiraga and Saito \cite{HS12} on the local Langlands correspondence and the $S$-groups. In order to study parabolic induction , we also have to generalize these results to the Levi subgroups.

In \S\ref{sec:dual-R}, we recapitulate Arthur's definition of dual $R$-groups via the omnipresent commutative diagram in Proposition \ref{prop:diagram}. The results obtained earlier can then be easily assembled, and Arthur's conjecture on $R$-groups for the inner forms of $\SL(N)$ (Theorem \ref{prop:main}) follows.

\paragraph{Acknowledgements}
The authors would like to thank Kwangho Choiy and David Goldberg for communicating their results to us, as well as the referee for meticulous reading. The second-named author is grateful to Dipendra Prasad for his interest in this work and helpful comments.

\section{Preliminaries}\label{sec:preliminaries}
\subsection{Conventions}
\paragraph{Local fields}
Throughout this article, $F$ always denotes a non-archimedean local field of characteristic zero. We set
\begin{itemize}
  \item $\Gamma_F$: the absolute Galois group of $F$, defined with respect to a chosen algebraic closure $\bar{F}$;
  \item $W_F$: the Weil group of $F$;
  \item $\WD_F := W_F \times \SU(2)$: the Weil-Deligne group of $F$;
  \item $|\cdot| = |\cdot|_F$: the normalized absolute value of $F$;
  \item $q_F$: the cardinality of the residue field of $F$.
\end{itemize}
When discussing the canonical family of normalizing factors for $\GL(N)$ and its inner forms, we will also fix a non-trivial additive character $\psi_F: F \to \C^\times$.

The usual Galois cohomology over $F$ is denoted by $H^\bullet(F,\cdot)$. Continuous cohomology of $W_F$ is denoted by $H^\bullet_\text{cont}(W_F, \cdots)$; the groups of continuous cocycles are denoted by $Z^\bullet_\text{cont}(W_F, \cdots)$.

\paragraph{Groups and representations}
For an $F$-group scheme $G$, the group of its $F$-points is denoted by $G(F)$; subgroups of $G$ mean the closed subgroup schemes. The identity connected component of $G$ is denoted by $G^0$. The center of $G$ is denoted by $Z_G$. Centralizers (resp. normalizers) in $G$ are denoted by $Z_G(\cdot)$ (resp. $N_G(\cdot)$). The algebraic groups over $\C$ are identified with their $\C$-points.

The derived group of $G$ is denoted by $G_\text{der}$. Now assume $G$ to be connected reductive. A simply connected cover of $G_\text{der}$, which is unique up to isomorphism, is denoted by $G_\text{SC} \to G_\text{der}$. We denote the adjoint group of $G$ by $G_\text{AD} := G/Z_G$. For every subgroup $H$ of $G$, we denote by $H_\text{sc}$ (resp. $H_\text{ad}$) the preimage of $H$ in $G_\text{SC}$ (resp. image in $G_\text{AD}$). The same formalism pertains to connected reductive $\C$-groups as well.

The definitions of the dual group $\Lgrp{G} = \hat{G} \rtimes W_F$ and the L-parameters will be reviewed in \S\ref{sec:res-L}.

The symbol $\Ad(\cdots)$ denotes the adjoint action of an abstract group on itself, namely $\Ad(x): g \mapsto gxg^{-1}$.

For any division algebra $D$ over $F$ and $n \in \Z_{\geq 1}$, we denote by $\GL_D(n)$ the group of invertible elements in $\End_D(D^n)$, where $D^n$ is viewed as a right $D$-module. It is also regarded as a connected reductive $F$-group.

The representations considered in this article are all over $\C$-vector spaces. For a connected reductive $F$-group $G$, we define
\begin{itemize}
  \item $\Pi(G)$: the set of equivalence classes of irreducible smooth representations of $G(F)$;
  \item $\Pi_{\text{unit}}(G)$: the subset consisting of unitary (i.e. unitarizable) representations;
  \item $\Pi_{\text{temp}}(G)$: the subset consisting of tempered representations;
  \item $\Pi_{2,\text{temp}}(G)$: the subset consisting of unitary representations which are square-integrable modulo the center.
\end{itemize}

For an abstract group $S$, we will also denote by $\Pi(S)$ the set of its irreducible representations up to equivalence.

The central character of $\pi \in \Pi(G)$ is denoted by $\omega_\pi$. The group of morphisms (resp. the set of isomorphisms) in the category of representations of $G(F)$ is denoted by $\Hom_G(\cdots)$ (resp. $\Isom_G(\cdots)$).

For any topological group $H$, we set
$$ H^D := \{\chi: H \to \C^\times, \text{ continuous character} \}. $$
For any representation $\pi$ of $G(F)$ and any $\eta \in G(F)^D$, we write $\eta\pi := \eta \otimes \pi$ for abbreviation. Also note that $\pi$ and $\eta\pi$ have the same underlying $\C$-vector spaces. If $M$ is a subgroup of $G$ and $\pi$ is a smooth representation of $G(F)$, we shall denote the restriction of $\pi$ to $M(F)$ by $\pi|_M$.

\paragraph{Combinatorics}
Let $G$ be a connected reductive $F$-group. We employ the following notations in this article. Let $M$ be a Levi subgroup, we write
\begin{itemize}
  \item $\mathcal{P}(M)$: the set of parabolic subgroups of $G$ with Levi component $M$;
  \item $\mathcal{L}(M)$: the set of Levi subgroups of $G$ containing $M$;
  \item $\mathcal{F}(M)$: the set of parabolic subgroups of $G$ containing $M$;
  \item $W(M) := N_G(M)(F)/M(F)$: the Weyl group (in a generalized sense) relative to $M$;
\end{itemize}
The Levi decompositions are written as $P=MU$, where $U$ denotes the unipotent radical of $P$. For $M$ chosen, the opposite parabolic of $P=MU$ is denoted by $\bar{P} = M\bar{U}$. When we have to emphasize the role of $G$, the notations $\mathcal{P}^G(M)$, $\mathcal{L}^G(M)$, $\mathcal{F}^G(M)$ and $W^G(M)$ will be used.

Let $w \in W(M)$ with a representative $\tilde{w} \in G(F)$. For $\sigma \in \Pi(M)$, we define $\tilde{w}\sigma$ to be the representation on the same underlying vector space, with the new action
$$ (\tilde{w}\sigma)(m) := \sigma(\tilde{w}^{-1} m \tilde{w}), \quad m \in M(F). $$
The equivalence class of $\tilde{w}\sigma$ depends only on $w \in W(M)$, and we will write $w\sigma$ instead, if there is no confusion.

Define $\mathcal{X}(G) := \Hom_{F-\text{grp}}(G, \Gm)$ and $\mathfrak{a}_G := \Hom(\mathcal{X}(G), \R)$. For every Levi subgroup $M$, there is a canonically split short exact sequence of finite-dimensional $\R$-vector spaces
$$ 0 \to \mathfrak{a}_G \to \mathfrak{a}_M \leftrightarrows \mathfrak{a}^G_M \to 0. $$
The linear duals of these spaces are denoted by $\mathfrak{a}^*_G$, etc. We also write $\mathfrak{a}_{G,\C} := \mathfrak{a}_G \otimes_\R \C$, etc. Sometime we also write $\mathfrak{a}_P$ instead of $\mathfrak{a}_M$ if $P=MU$.

The Harish-Chandra map $H_G: G(F) \to \mathfrak{a}_G$ is the homomorphism characterized by
$$ \angles{\chi, H_G(x)} = \log|\chi(x)|_F, \quad \chi \in \mathcal{X}(G). $$
For $\lambda \in \mathfrak{a}^*_{G,\C}$ and $\pi \in \Pi(G)$, we define $\pi_\lambda \in \Pi(G)$ by
\begin{gather}\label{eqn:pi-twist}
  \pi_\lambda := e^{\angles{\lambda, H_G(\cdot)}} \otimes \pi.
\end{gather}

Fix a minimal parabolic subgroup $P_0 = M_0 U_0$ of $G$. We define $\Delta_0$, $\Delta_0^\vee$ to be the set of simple roots and coroots, which form bases of $(\mathfrak{a}^G_{M_0})^*$ and $\mathfrak{a}^G_{M_0}$, respectively. The set of positive roots is denoted by $\Sigma_0$, and its subset of reduced roots by $\Sigma^{\text{red}}_0$. They form a bona fide root system. For every $P = MU \supset P_0$, we define $\Delta_P \subset \Sigma_P \subset \Sigma^{\text{red}}_P$ by taking the set of nonzero restrictions to $(\mathfrak{a}^G_M)^*$ of elements in  $\Delta_0 \subset \Sigma_0 \subset \Sigma^\text{red}_0$. To each $\alpha \in \Sigma_P$ we may associate the coroot $\alpha^\vee \in \mathfrak{a}^G_M$: it is defined as the restriction of the coroot in $\Delta^\vee_0$. For a given $P$, the objects above are independent of the choice of $P_0$. We can emphasize the role of $G$ by using the notations $\Delta^G_P$, etc. whenever needed.

\paragraph{Induction}
We always consider a parabolic subgroup $P=MU$ of $G$. The modulus character of $P(F)$ is denoted by $\delta_P$, i.e.
$$ (\text{left Haar measure}) = \delta_P \cdot (\text{right Haar measure}). $$

The usual smooth induction functor is denoted by $\Ind(\cdots)$. The normalized parabolic induction functor from $P$ to $G$ is denoted by $I^G_P(\cdot) := \Ind^G_P(\delta^{\frac{1}{2}}_P \otimes \cdot)$. Recall that for $\sigma \in \Pi(M)$ with underlying vector space $V_\sigma$, we use the usual model to realize $I^G_P(\sigma)$ as the space of functions $\varphi: G(F) \to V_\sigma$ such that $\varphi$ is invariant under right translation by an open compact subgroup of $G(F)$, and that $\varphi(umx) = \delta_P(m)^{\frac{1}{2}}\sigma(m)(\varphi(x))$ for all $m \in M(F)$, $u \in U(F)$. The group $G(F)$ acts on this function space by the right regular representation.

For $\sigma, \sigma' \in \Pi(M)$ and $f \in \Hom_M(\sigma, \sigma')$, the induced morphism is denoted by $I^G_P(f)$; it sends $\varphi$ to $f(\varphi)$.

\subsection{Normalized intertwining operators}
Our basic reference for normalized intertwining operators is \cite{Ar93}. Consider the following data
\begin{itemize}
  \item $G$: a connected reductive $F$-group.
  \item $M$: a Levi subgroup of $G$,
  \item $P,Q \in \mathcal{P}(M)$,
  \item $\sigma: M(F) \to \Aut_\C(V_\sigma)$: a smooth representation of $M(F)$ of finite length,
  \item $\lambda \in \mathfrak{a}^*_{M,\C}$.
\end{itemize}

For every $\alpha \in \Delta_P$, we denote by $r_\alpha$ the smallest positive rational number such that $r_\alpha \cdot \alpha^\vee$ lies in the lattice $H_M(M(F))$. We define
\begin{gather}\label{eqn:alpha-check}
  \check{\alpha} := r_\alpha \alpha^\vee .
\end{gather}

By recalling \eqref{eqn:pi-twist}, we form the normalized parabolic induction $I^G_P(\sigma_\lambda)$, $I^G_Q(\sigma_\lambda)$. Their underlying spaces are denoted by $I^G_P(V_{\sigma_\lambda})$, $I^G_Q(V_{\sigma_\lambda})$. The standard intertwining operator
$$ J_{Q|P}(\sigma_\lambda): I^G_P(\sigma_\lambda) \longrightarrow I^G_Q(\sigma_\lambda) $$
is defined by the absolutely convergent integral
\begin{gather}\label{eqn:J-int}
  (J_{Q|P}(\sigma_\lambda)\varphi)(x) = \int_{U_P(F) \cap U_Q(F) \backslash U_Q(F)} \varphi(u x) \dd u, \quad x \in G(F),
\end{gather}
when $\angles{\Re(\lambda), \alpha^\vee } \gg 0$ for all $\alpha \in \Sigma^\text{red}_P \cap \Sigma^\text{red}_{\bar{Q}}$; see \cite[IV.1]{Wa03} for the precise meaning of absolute convergence. Recall that upon choosing a special maximal compact open subgroup $K \subset G(F)$ in good position relative to $M$, these induced representations can be realized on a vector space that is independent of $\lambda$. It is known that $J_{Q|P}(\sigma_\lambda)$ is a rational function in the variables
$$ \left\{ q_F^{-\angles{\lambda, \check{\alpha}}} : \alpha \in \Delta_P \right\}. $$
In particular, as a function in $\lambda$, $J_{Q|P}(\sigma_\lambda)$ admits a meromorphic continuation to $\mathfrak{a}^*_{M,\C}$. When $\sigma \in \Pi_{\text{temp}}(M)$, it is known that \eqref{eqn:J-int} is absolutely convergent for $\angles{\Re(\lambda), \alpha^\vee} > 0$ for all $\alpha \in \Sigma^\text{red}_P \cap \Sigma^\text{red}_{\bar{Q}}$. Moreover, as a meromorphic family of operators, it satisfies $\text{ord}_{\lambda=0} (J_{Q|P}(\sigma_\lambda)) \geq -1$.

Henceforth we assume $\sigma$ irreducible, i.e. $\sigma \in \Pi(M)$. Take any $P \in \mathcal{P}(M)$, define the $j$-functions as
\begin{gather}\label{eqn:j-function}
  j(\sigma_\lambda) := J_{P|\bar{P}}(\sigma_\lambda) J_{\bar{P}|P}(\sigma_\lambda).
\end{gather}

It is known that $\lambda \mapsto j(\sigma_\lambda)$ a scalar-valued meromorphic function, which is not identically zero. Moreover, $j(\sigma_\lambda)$ is independent of $P$ and admits a product decomposition
$$ j(\sigma_\lambda) = \prod_{\alpha \in \Sigma^\text{red}_P} j_\alpha(\sigma_\lambda) $$
where $j_\alpha$ denotes the $j$-function defined relative to the Levi subgroup $M_\alpha \in \mathcal{L}(M)$ such that $\Sigma^{M_\alpha,\text{red}}_M = \{\pm \alpha\}$.

Now assume $\sigma \in \Pi_{2,\text{temp}}(M)$. In this paper, we define Harish-Chandra's $\mu$-function as the meromorphic function
$$ \mu(\sigma_\lambda) := j(\sigma_\lambda)^{-1}. $$
Accordingly, $\mu$ also admits a product decomposition $\mu = \prod_\alpha \mu_\alpha$. It is analytic and non-negative for $\lambda \in i\mathfrak{a}^*_M$. Note that our definitions of $j$-functions and $\mu$-functions depend on the choice of Haar measures on unipotent radicals. In particular, our $\mu$-function differs from that in \cite[V.2]{Wa03} by some harmless constant.

\begin{definition}[cf. {\cite[\S 2]{Ar89-IOR1}}]\label{def:normalizing}
  In this article, a family of normalizing factors is a family of meromorphic functions on the $\mathfrak{a}^*_{M,\C}$-orbits in $\Pi(M)$, for all Levi subgroup $M$ of $G$, written as
  $$ r_{Q|P}(\sigma_\lambda), \quad P,Q \in \mathcal{P}(M), \sigma_\lambda \in \mathfrak{a}^*_{M,\C} $$
  satisfying the following conditions. First of all, we define the corresponding normalized intertwining operators as
  $$ R_{Q|P}(\sigma_\lambda) := r_{Q|P}(\sigma_\lambda)^{-1} J_{Q|P}(\sigma_\lambda) $$
  which is a meromorphic family (in $\lambda$) of intertwining operators $I^G_P(\sigma_\lambda) \to I^G_Q(\sigma_\lambda)$.

  We shall also assume that a family of normalizing factors is chosen for every proper Levi subgroup.
  \renewcommand{\labelenumi}{$\mathbf{R}_\arabic{enumi}$}\begin{enumerate}
    \item For all $P, P', P'' \in \mathcal{P}(M)$, we have $R_{P''|P}(\sigma_\lambda) = R_{P''|P'}(\sigma_\lambda) R_{P'|P}(\sigma_\lambda)$.
    \item If $\sigma \in \Pi_\text{unit}(M)$, then
      $$ R_{Q|P}(\sigma_\lambda) = R_{P|Q}(\sigma_{-\bar{\lambda}})^*, \quad \lambda \in \mathfrak{a}_{M,\C}^* . $$
      In particular, $R_{Q|P}(\sigma)$ is a well-defined unitary operator.
    \item This family is compatible with conjugacy, namely
      $$ R_{gQg^{-1}|gPg^{-1}}(g\sigma_\lambda) = \ell(g) R_{Q|P}(\sigma_\lambda) \ell(g)^{-1} $$
      for all $g \in G(F)$, where $\ell(g)$ is the map $\varphi(\cdot) \mapsto \varphi(g^{-1} \cdot)$
    \item We have
      $$ r_{Q|P}(\sigma_\lambda) = \prod_{\alpha \in \Sigma_P^\text{red} \cap \Sigma_{\bar{Q}}^\text{red}} r^{\tilde{M}_\alpha}_{\overline{P_\alpha}|P_\alpha}(\sigma_\lambda), $$
      where $P_\alpha := P \cap M_\alpha$, and $r^{M_\alpha}_{\overline{P_\alpha}|P_\alpha}$ comes from the family of normalizing factors for $M_\alpha$.
    \item Let $S=LU \in \mathcal{F}(M)$ containing both $P$ and $Q$, then $R_{Q|P}(\sigma_\lambda)$ is the operator deduced from $R^L_{P \cap L|Q \cap L}(\sigma_\lambda)$ by the functor $I^G_S(\cdot)$.
    \item The function $\lambda \mapsto r_{Q|P}(\sigma_\lambda)$ is rational in the variables $\left\{ q_F^{-\angles{\lambda, \check{\alpha}}} : \alpha \in \Delta_P \right\}$.
    \item If $\sigma \in \Pi_\text{temp}(M)$, then the meromorphic function $\lambda \mapsto r_{Q|P}(\sigma_\lambda)$ is invertible whenever $\Re\angles{\lambda, \alpha^\vee} > 0$ for all $\alpha \in \Delta_P$.
  \renewcommand{\labelenumi}{\arabic{enumi}}\end{enumerate}
\end{definition}

Observe that $\mathbf{R}_2$ is equivalent to that $r_{Q|P}(\sigma_\lambda) = \overline{r_{P|Q}(\sigma_{-\bar{\lambda}})}$ for $\sigma \in \Pi_\text{unit}(M)$, as the unnormalized operators $J_{Q|P}(\sigma_\lambda)$ satisfy a similar condition. Similarly, $\mathbf{R}_3$ is equivalent to that $r_{gQg^{-1}|gPg^{-1}}(g\sigma_\lambda) = r_{Q|P}(\sigma_\lambda)$.

The fundamental result about the normalizing factors is that they exist \cite[Theorem 2.1]{Ar89-IOR1}.

\begin{remark}\label{rem:normalization}
  According to Langlands \cite[Appendix 2]{Lan76}, there is a conjectural canonical family of normalizing factors $r_{Q|P}(\sigma_\lambda)$ in terms of local factors, namely
  $$ r_{Q|P}(\sigma_\lambda) = \varepsilon(0, \rho_{Q|P}^\vee \circ \phi_\lambda, \psi_F)^{-1} L(0, \phi_\lambda, \rho_{Q|P}^\vee) L(1, \phi_\lambda, \rho_{Q|P}^\vee)^{-1}, $$
  where
  \begin{itemize}
    \item $\phi_\lambda$ is the Langlands parameter for $\sigma_\lambda$;
    \item let $\hat{\mathfrak{u}}_Q$ (resp. $\hat{\mathfrak{u}}_P$) denote the Lie algebra of the unipotent radical of the dual parabolic subgroup $\hat{Q}$ (resp. $\hat{P}$) in $\hat{G}$;
    \item $\rho_{Q|P}$ is the adjoint representation of $\Lgrp{M}$ on $\hat{\mathfrak{u}}_Q/(\hat{\mathfrak{u}}_Q \cap \hat{\mathfrak{u}}_P)$, and $\rho_{Q|P}^\vee$ denotes its contragredient;
    \item $\psi_F: F \to \C^\times$ is a chosen non-trivial additive character.
  \end{itemize}

  We will invoke this description only in the case $G=\GL_F(n)$. In that case, the local factors in sight are essentially those associated with pairs $(\phi_1, \phi_2)$ where $\phi_1, \phi_2$ are among the L-parameters parametrizing the components of $\sigma$. Such Artin local factors are known to agree with their representation-theoretic avatars, say those defined by Rankin-Selberg convolution or by the Langlands-Shahidi method.
\end{remark}

\begin{remark}\label{rem:construction-normalization}
  The construction of normalizing factors can be reduced to the case that $M$ is a maximal proper Levi subgroup of $G$ and $\sigma \in \Pi_{2,\text{temp}}(M)$, as illustrated in \cite{Ar89-IOR1}. Let us give a quick sketch of this reduction.
  \begin{enumerate}
    \item In view of $\mathbf{R}_4$, we are led to the case $M$ maximal proper. Moreover, it suffices to verify $\mathbf{R}_3$ for the representatives in $G(F)$ of the elements in $W(M) := N_{G(F)}(M)/M(F)$, which has at most two elements.
    \item Assume that $\sigma \in \Pi_\text{temp}(M)$. By the classification of tempered representations, there exist a parabolic subgroup $R=M_R U_R$ of $M$ and $\tau \in \Pi_\text{temp}(M_R)$ such that $\sigma \hookrightarrow I^M_R(\tau)$. The pair $(M,\tau)$ is unique up to conjugacy. There is a unique element $P(R)$ in $\mathcal{P}(M_R)$, characterized by the properties
      \begin{itemize}
        \item $P(R) \subset P$,
        \item $P(R) \cap M = R$.
      \end{itemize}
      Consequently, parabolic induction in stages gives $I^G_P I^M_R(\tau) = I^G_{P(R)}(\tau)$. The same construction works when $P$ is replaced by $Q$. Set
      $$ r_{Q|P}(\sigma_\lambda) := r_{Q(R)|P(R)}(\tau_\lambda). $$

      In view of $\mathbf{R}_5$ together with parabolic induction in stages, we see that $R_{Q|P}(\sigma_\lambda)$ is the restriction of $R_{Q(R)|P(R)}(\sigma_\lambda)$ to $I^G_P(\sigma_\lambda)$. The required conditions can be readily verified.
    \item For general $\sigma$, we may realize it as the Langlands quotient $I^M_R(\tau_\mu) \twoheadrightarrow \sigma$, where $R = M_R U_R$ is a parabolic subgroup of $M$, $\tau \in \Pi_\text{temp}(M_R)$ and $\mu \in \mathfrak{a}^*_{M_R}$ satisfies $\Re\angles{\mu, \beta^\vee} > 0$ for all $\beta \in \Delta^M_R$. The triplet $(M,\tau,\mu)$ is again unique up to conjugacy. Let $P,Q \in \mathcal{P}(M)$, define $P(R), Q(R) \in \mathcal{P}(M_R)$ as before and set
    $$ r_{Q|P}(\sigma_\lambda) := r_{Q(R)|P(R)}(\tau_{\lambda+\mu}). $$

      Recall that
      $$ \Ker[I^M_R(\tau_\mu) \twoheadrightarrow \sigma] = \Ker(J^M_{\bar{R}|R}(\tau_\mu)). $$

      The condition $\mathbf{R}_7$ applied to $\tau_\mu$ tells us that $\Ker(J^M_{\bar{R}|R}(\tau_\mu)) = \Ker(R^M_{\bar{R}|R}(\tau_\mu))$. Idem for $\mu$ replaced by $\mu+\lambda$. Using $\mathbf{R}_5$, one sees that $R_{Q(R)|P(R)}(\tau_{\lambda+\mu})$ factors into $R_{Q|P}(\sigma_\lambda)$ on $I^G_P(\sigma_\lambda)$. All conditions except $\mathbf{R}_2$ follow from this. The proof of $\mathbf{R}_2$ requires somehow more efforts to deal with the unitarizability of Langlands quotients; the reader can consult \cite[p.30]{Ar89-IOR1} for details.
    \item Reverting to our original assumption that $M$ is maximal proper and $\sigma \in \Pi_{2,\text{temp}}(M)$, it is clear that it remains to verify conditions $\mathbf{R}_1$, $\mathbf{R}_2$, $\mathbf{R}_3$, $\mathbf{R}_6$, $\mathbf{R}_7$. Furthermore, one can reduce $\mathbf{R}_3$ to the assertion that $r_{w\bar{P}w^{-1}|wPw^{-1}}(w(\sigma_\lambda)) = r_{\bar{P}|P}(\sigma_\lambda)$, for $w \in W(G)$ being the non-trivial element in $W(M)$ if it exists.
  \end{enumerate}
\end{remark}

\subsection{Knapp-Stein $R$-groups}
Fix a family of normalizing factors for $G$. Assume henceforth that $M$ is a Levi subgroup of $G$ and $\sigma \in \Pi_{2,\text{temp}}(M)$. Define the isotropy group
$$ W_\sigma := \{w \in W(M) : w\sigma \simeq \sigma \}. $$

Fix $P \in \mathcal{P}(M)$. For $w \in W(M)$ with a representative $\tilde{w} \in G(F)$, we define the operator $r_P(\tilde{w}, \sigma) \in \Isom_G(I^G_P(\sigma), I^G_P(\tilde{w}\sigma))$ by
\begin{gather}\label{eqn:r_P}
  r_P(\tilde{w}, \sigma): I^G_P(\sigma) \xrightarrow{R_{w^{-1}Pw|P}(\sigma)} I^G_{w^{-1}Pw}(\sigma) \xrightarrow{[\ell(\tilde{w}): \phi \mapsto \phi(\tilde{w}^{-1} \cdot)]} I^G_P(\tilde{w}\sigma).
\end{gather}
We notice the property that for any $w, w' \in W(M)$ with representatives $\tilde{w}, \tilde{w}' \in G(F)$, we have
\begin{gather}\label{eqn:r_P-property}
  r_P(\tilde{w}\tilde{w}', \sigma) = r_P(\tilde{w}, \tilde{w}'\sigma) \circ r_P(\tilde{w}', \sigma).
\end{gather}

Assume now $w \in W_\sigma$. Choose a representative $\tilde{w}$ of $w$ and $\sigma(\tilde{w}) \in \Isom_M(\tilde{w}\sigma, \sigma)$ to define the operator
$$ R_P(\tilde{w}, \sigma) := I^G_P(\sigma(\tilde{w})) \circ r_P(\tilde{w}, \sigma). $$

The class $R_P(\tilde{w}, \sigma) \text{ mod } \C^\times$ is independent of the choices of $\sigma(\tilde{w})$ and the representative $\tilde{w}$. We also have
\begin{gather*}
  R_P(\tilde{w}, \sigma) \in \Aut_G(I^G_P(\sigma)), \\
  R_P(\tilde{w}, \sigma) R_P(\tilde{w}', \sigma) = R_P(\tilde{w} \tilde{w}', \sigma) \mod \C^\times, \quad w, w' \in W_\sigma.
\end{gather*}

Now we can define the Knapp-Stein $R$-group as follows.
\begin{align*}
  W^0_\sigma & := \{w \in W_\sigma : R_P(\tilde{w},\sigma) \in \C^\times \identity \}, \\
  R_\sigma & := W_\sigma/W^0_\sigma.
\end{align*}

We will also make use of the following alternative description of $R_\sigma$.
\begin{proposition}
  The subgroup $W^0_\sigma$ is the Weyl group of the root system on $\mathfrak{a}_M$ composed of the multiples of the roots in $\{ \alpha \in \Sigma^\mathrm{red}_P : \mu_\alpha(\sigma)=0 \}$.

  Given any Weyl chamber $\mathfrak{a}^+_\sigma \subset \mathfrak{a}_M$ for the aforementioned root system, there is then a unique section $R_\sigma \hookrightarrow W_\sigma$ that sends $r \in R_\sigma$ to the representative $w \in W_\sigma$ such that $w\mathfrak{a}^+_\sigma = \mathfrak{a}^+_\sigma$. Consequently, we can write $W_\sigma = W^0_\sigma \rtimes R_\sigma$.
\end{proposition}
In the literature, $R_\sigma$ is sometimes viewed as a subgroup of $W_\sigma$ in this manner, eg. \cite{Go06}.

Write $V_\sigma$ (resp. $I^G_P(V_\sigma)$) for the underlying vector space of the representation $\sigma$ (resp. $I^G_P(\sigma)$). It follows that $w \mapsto R_P(\tilde{w}, \sigma)$ induces a projective representation of $R_\sigma$ on $I^G_P(V_\sigma)$, where $\tilde{w} \in G(F)$ is any representative of $w \in W_\sigma$. We denote this projective representation provisionally by $r \mapsto R_P(r, \sigma)$, for $r \in R_\sigma$.

There is a standard way to lift $R_P(\cdot,\sigma)$ to an authentic representation of some group $\tilde{R}_\sigma$ which sits in a central extension
$$ 1 \to \C^\times \to \tilde{R}_\sigma \to R_\sigma \to 1, $$
such that $\C^\times$ acts by $z \mapsto z \cdot \identity$. Namely, we can set $\tilde{R}_\sigma$ to be the group of elements $(r, M[r]) \in R_\sigma \times \Aut_\C(I^G_P(V_\sigma))$ such that $M[r] \text{ mod } \C^\times$ gives $R_P(r, \sigma)$. The lifted representation, denoted by $\tilde{r} \mapsto R_P(\tilde{r},\sigma)$, is then $\tilde{r} = (r, M[r]) \mapsto M[r]$. Such a central extension by $\C^\times$ that lifts $R_P(\cdot, \sigma)$ is unique up to isomorphism.

Note that the central extension above can also be described by the class $\mathbf{c}_\sigma \in H^2(R_\sigma, \C^\times)$ coming from the $\C^\times$-valued $2$-cocycle $c_\sigma$ defined by
\begin{gather}\label{eqn:R-cocycle}
  R_P(\widetilde{r_1 r_2}, \sigma) = c_\sigma(r_1, r_2) R_P(\tilde{r}_1, \sigma) R_P(\tilde{r}_2, \sigma), \quad r_1, r_2 \in R_\sigma
\end{gather}
where we choose a preimage $\tilde{r} \in \tilde{R}_\sigma$ for every $r \in R_\sigma$.

\begin{theorem}[Harish-Chandra, Silberger \cite{Si78}]
  Fix a preimage $\tilde{r} \in \tilde{R}_\sigma$ for every $r \in R_\sigma$, then the operators $\{ R_P(\tilde{r}, \sigma) : r \in R_\sigma \}$ form a basis of $\End_G(I^G_P(\sigma))$.
\end{theorem}

Following Arthur, we reformulate this fundamental result as follows. Let
\begin{align*}
  \Pi_\sigma(G) & := \{ \text{irreducible constituents of } I^G_P(\sigma) \}/\simeq , \\
  \Pi_-(\tilde{R}_\sigma) & := \{ \rho \in \Pi(\tilde{R}_\sigma) : \forall z \in \C^\times, \; \rho(z) = z\cdot\identity \}.
\end{align*}
Note that $\Pi_\sigma(G)$, $\Pi_-(\tilde{R}_\sigma)$ are both finite sets, and each $\rho \in \Pi_-(\tilde{R}_\sigma)$ is finite-dimensional.

\begin{corollary}
  Let $\mathcal{R}$ be the representation of $\tilde{R}_\sigma \times G(F)$ on $I^G_P(V_\sigma)$ defined by
  $$ \mathcal{R}(\tilde{r}, x) = R_P(\tilde{r}, \sigma) I^G_P(\sigma,x), \quad \tilde{r} \in \tilde{R}_\sigma, \; x \in G(F). $$
  Then there is a decomposition
  \begin{gather}\label{eqn:R-decomp}
    \mathcal{R} \simeq \bigoplus_{\rho \in \Pi_-(\tilde{R}_\sigma)} \rho \boxtimes \pi_\rho,
  \end{gather}
  where $\rho \mapsto \pi_\rho$ is a bijection from $\Pi_-(\tilde{R}_\sigma)$ to $\Pi_\sigma(G)$, characterized by \eqref{eqn:R-decomp}.
\end{corollary}

Consequently, $I^G_P(\sigma)$ is irreducible if and only if $R_\sigma=\{1\}$.

\begin{remark}
  When $G$ is quasisplit and $\sigma$ is generic with respect to a given Whittaker datum for $M$, the work of Shahidi \cite{Sh90} furnishes
  \renewcommand{\labelenumi}{(\roman{enumi})}\begin{enumerate}
    \item a canonical family of normalizing factors $r_{Q|P}(\sigma)$;
    \item a canonically defined homomorphism $w \mapsto R_P(w, \sigma)$;
    \item a canonical splitting of the central extension $1 \to \C^\times \to \tilde{R}_\sigma \to R_\sigma \to 1$.
  \end{enumerate}\renewcommand{\labelenumi}{\arabic{enumi}.}

  These properties are not expected in general. Indeed, we shall see in Example \ref{ex:nonsplit} that (iii) may fail.
\end{remark}

\begin{remark}
  The formalism above depends not only on $(M,\sigma)$, but also on the choice of $P \in \mathcal{P}(M)$. One can easily pass to another choice $Q \in \mathcal{P}(M)$ by transport of structure using $R_{Q|P}(\sigma)$. For example, one has
  \begin{align*}
    r_Q(\tilde{w}, \sigma) & = R_{P|Q}(\tilde{w}\sigma)^{-1} r_P(\tilde{w}, \sigma) R_{P|Q}(\sigma), \\
    R_Q(\tilde{w}, \sigma) & = R_{P|Q}(\sigma)^{-1} R_P(\tilde{w}, \sigma) R_{P|Q}(\sigma)
  \end{align*}
  for all $w \in W(M)$ with a representative $\tilde{w} \in G(F)$ and some chosen $\sigma(\tilde{w})$.
\end{remark}

\section{Restriction}\label{sec:res}
Let $G$, $G^\sharp$ be connected reductive $F$-groups such that
$$ G_\text{der} \subset G^\sharp \subset G. $$

\subsection{Restriction of representations}\label{sec:res-rep}
In this subsection, we will review the basic results in \cite[\S 2]{Ta92} and \cite[Chapter 2]{HS12} concerning the restriction of a smooth representation from $G(F)$ to $G^\sharp(F)$. The objects associated to $G^\sharp$ are endowed with the superscript $\sharp$, eg. $\pi^\sharp \in \Pi(G^\sharp)$.

\begin{proposition}[\cite{Si79}, {\cite[Lemma 2.1 and Proposition 2.2]{Ta92}}]\label{prop:lifting}
  Let $\pi \in \Pi(G)$, then $\pi|_{G^\sharp}$ decomposes into a finite direct sum of smooth irreducible representations. Each irreducible constituent of $\pi|_{G^\sharp}$ has the same multiplicity.

  Conversely, every $\pi^\sharp \in \Pi(G^\sharp)$ embeds into $\pi|_{G^\sharp}$ for some $\pi \in \Pi(G)$. If the central character $\omega_{\pi^\sharp}$ is unitary, one can choose $\pi$ so that $\omega_\pi$ is also unitary.
\end{proposition}

\begin{proposition}[{\cite[Corollary 2.5]{Ta92}}]\label{prop:res-disjoint}
  Let $\pi_1, \pi_2 \in \Pi(G)$. The following are equivalent:
  \begin{enumerate}
    \item $\Hom_{G^\sharp}(\pi_1, \pi_2) \neq \{0\}$;
    \item $\pi_1|_{G^\sharp} \simeq \pi_2|_{G^\sharp}$;
    \item there exists $\eta \in (G(F)/G^\sharp(F))^D$ such that $\eta \pi_1 \simeq \pi_2$.
  \end{enumerate}
\end{proposition}

For $\pi \in \Pi(G)$, we define a finite ``packet'' of smooth irreducible representations of $G^\sharp(F)$ as
$$ \Pi_\pi := \{\text{irreducible constituents of } \pi|_{G^\sharp} \}/\simeq . $$
Consequently, Proposition \ref{prop:res-disjoint} implies that $\Pi(G^\sharp) = \bigsqcup_\pi \Pi_\pi$, when $\pi$ is taken over the $(G(F)/G^\sharp(F))^D$-orbits in $\Pi(G)$.

\begin{proposition}[{\cite[Proposition 2.7]{Ta92}}]\label{prop:heredity}
  Let $\pi \in \Pi(G)$ and assume that $\omega_\pi$ is unitary. Let $\mathbf{P}$ be one of the following properties of smooth irreducible representations of $G(F)$ or $G^\sharp(F)$:
  \begin{enumerate}
    \item unitary,
    \item tempered,
    \item square-integrable modulo the center,
    \item cuspidal.
  \end{enumerate}
  Then we have equivalences of the form
  $$ [ \pi \text{ satisfies } \mathbf{P} ] \Leftrightarrow [\exists \pi^\sharp \in \Pi_\pi, \; \pi^\sharp \text{ satisfies } \mathbf{P} ] \Leftrightarrow [\forall \pi^\sharp \in \Pi_\pi, \; \pi^\sharp \text{ satisfies } \mathbf{P} ]. $$
\end{proposition}

Now comes the decomposition of $\pi|_{G^\sharp}$.

\begin{definition}
  Let $\pi \in \Pi(G)$ with the underlying $\C$-vector space $V_\pi$. Note that $V_{\eta\pi} = V_\pi$ for all $\eta \in (G(F)/G^\sharp(F))^D$. Introduce the following groups
  \begin{align*}
    X^G(\pi) & := \{\eta \in (G(F)/G^\sharp(F))^D : \eta\pi \simeq \pi \}, \\
    S^G(\pi) & := \langle I^G_\eta \in \Isom_G(\eta\pi, \pi) : \eta \in X^G(\pi) \rangle \; \subset \Aut_{G^\sharp}(\pi).
  \end{align*}
  Observe that $\Isom_G(\eta\pi, \pi)$ is a $\C^\times$-torsor by Schur's lemma, and an element $I^G_\eta \in S^G(\pi)$ uniquely determines $\eta$. The group law is given by composition in $\Aut_\C(V_\pi)$, namely by
  $$ \Isom_G(\eta\pi, \pi) \times \Isom_G(\eta'\pi, \pi) = \Isom_G(\eta'\eta\pi, \eta'\pi) \times \Isom_G(\eta'\pi, \pi) \to \Isom_G(\eta'\eta\pi, \pi) $$
  for all $\eta, \eta' \in X^G(\pi)$.

  Thus we obtain a central extension of locally compact groups
  \begin{gather}\label{eqn:res-S}
    1 \to \C^\times \to S^G(\pi) \to X^G(\pi) \to 1,
  \end{gather}
  where the first arrow is $z \mapsto z \cdot \identity$ and the second one is $I^G_\eta \mapsto \eta$.
\end{definition}

Also note that $X^G(\pi)=X^G(\xi\pi)$, $S^G(\pi)=S^G(\xi\pi)$ for any $\xi \in (G(F)/G^\sharp(F))^D$.

\textit{Attention}: implicit in the notations above is the reference to $G^\sharp$, which is usually clear from the context. Indications to $G^\sharp$ will be given when necessary.

It is easy to see that $X^G(\pi)$ is finite abelian. As in the setting of $R$-groups, we define the finite set
$$ \Pi_-(S^G(\pi)) := \left\{ \rho \in \Pi(S^G(\pi)) : \forall z \in \C^\times, \; \rho(z) = z\cdot\identity \right\}. $$

\begin{theorem}[{\cite[Lemma 2.5 and Corollary 2.7]{HS12}}]\label{prop:S-decomp}
  Let $\mathfrak{S} = \mathfrak{S}(\pi)$ be the representation of $S^G(\pi) \times G^\sharp(F)$ on $V_\pi$ defined by
  $$ \mathfrak{S}(I,x) = I \circ \pi(x), \quad I \in S^G(\pi), x \in G^\sharp(F). $$
  Then there is a decomposition
  \begin{gather}\label{eqn:S-decomp}
    \mathfrak{S} \simeq \bigoplus_{\rho \in \Pi_-(S^G(\pi))} \rho \boxtimes \pi^\sharp_\rho,
  \end{gather}
  where $\rho \mapsto \pi^\sharp_\rho$ is a bijection from $\Pi_-(S^G(\pi))$ to $\Pi_\pi$, characterized by \eqref{eqn:S-decomp}.
\end{theorem}

\subsection{Relation to parabolic induction}\label{sec:res-ind}
Let $P$ be a parabolic subgroup of $G$ with a Levi decomposition $P=MU$. In this article, we denote systematically
\begin{align*}
  P^\sharp & := P \cap G^\sharp, \\
  M^\sharp & := M \cap M^\sharp.
\end{align*}
Then $P^\sharp$ is a parabolic subgroup of $G^\sharp$ with Levi decomposition $P^\sharp = M^\sharp U$, since every unipotent subgroup of $G$ is contained in $G_\text{der}$. The map $P \mapsto P^\sharp$ (resp. $M \mapsto M^\sharp$) induces a bijection between the parabolic subgroups (resp. Levi subgroups) of $G$ and $G^\sharp$, which leaves the unipotent radicals intact. We also have a canonical identification $W(M^\sharp) = W(M)$. In what follows, we will fix Haar measures on the unipotent radicals of parabolic subgroups of $G$ and $G^\sharp$, which are compatible with the identifications above.

Obviously, the modulus functions satisfy $\delta_P(m) = \delta_{P^\sharp}(m)$ for all $m \in M^\sharp(F)$.

\begin{lemma}[{\cite[Lemma 1.1]{Ta92}}]\label{prop:res-induction}
  Let $\sigma \in \Pi(M)$. Then we have the following isomorphism between smooth representations of $G^\sharp(F)$
  \begin{align*}
    I^G_P(\sigma)|_{G^\sharp} & \longrightarrow I^{G^\sharp}_{P^\sharp}(\sigma|_{M^\sharp}) \\
    \varphi & \longmapsto \varphi|_{G^\sharp(F)}
  \end{align*}
  which is functorial in $\sigma$.
\end{lemma}
\begin{proof}
  Upon recalling the definitions of $I^G_P(\sigma)$ and $I^{G^\sharp}_{P^\sharp}(\sigma|_{M^\sharp})$ as function spaces, the assertion follows from the canonical isomorphisms
  $$ P^\sharp(F) \backslash G^\sharp(F) = (P^\sharp \backslash G^\sharp)(F) \rightiso (P \backslash G)(F) = P(F) \backslash G(F) $$
  and the fact that $\delta_P|_{M^\sharp} = \delta_{P^\sharp}$.
\end{proof}

The next result will not be used in this article; we include it only for the sake of completeness Recall that the normalized Jacquet functor $r^G_P$ is the left adjoint of $I^G_P$. Idem for $r^{G^\sharp}_{P^\sharp}$.
\begin{lemma}\label{prop:res-Jacquet}
  Let $\pi \in \Pi(G)$. The restriction of representations induces an isomorphism
  $$ r^G_P(\pi)|_{M^\sharp} \rightiso r^{G^\sharp}_{P^\sharp}(\pi|_{G^\sharp}) $$
  between smooth representations of $M^\sharp(F)$, which is functorial in $\pi$.
\end{lemma}
\begin{proof}
  Evident.
\end{proof}

\begin{proposition}
  Let $M$ be a Levi subgroup of $G$. Then the inclusion map $M \hookrightarrow G$ induces an isomorphism between locally compact abelian groups
  $$ M(F)/M^\sharp(F) \rightiso G(F)/G^\sharp(F). $$
\end{proposition}
\begin{proof}
  The inclusion map induces an isomorphism $M/M^\sharp \hookrightarrow G/G^\sharp$ as $F$-tori. Hence the short exact sequence $1 \to M^\sharp \to M \to M/M^\sharp \to 1$ and its avatar for $G$ provide a commutative diagram of pointed sets with exact rows
  $$\xymatrix{
    1 \ar[r] & G^\sharp(F) \ar[r] & G(F) \ar[r] & (G/G^\sharp)(F) \ar[r] & H^1(F, G^\sharp) \\
    1 \ar[r] & M^\sharp(F) \ar[u] \ar[r] & M(F) \ar[u] \ar[r] & (M/M^\sharp)(F) \ar@{=}[u] \ar[r] & H^1(F, M^\sharp) \ar[u]
  }.$$
  Fix a parabolic subgroup $P$ of $G$ with a Levi decomposition $P=MU$. The rightmost vertical arrow factorizes as
  $$ H^1(F, M^\sharp) \to H^1(F,P^\sharp) \to H^1(F, G^\sharp).$$

  The first map is an isomorphism whose inverse is induced by $P^\sharp \twoheadrightarrow P^\sharp/U = M^\sharp$. It is well known that the second map is injective, hence so is the composition. A simple diagram chasing shows that $M(F)/M^\sharp(F) \rightiso G(F)/G^\sharp(F)$, as asserted.
\end{proof}

\begin{corollary}\label{prop:char-G-M}
  The restriction map induces an isomorphism
  $$ (G(F)/G^\sharp(F))^D \rightiso (M(F)/M^\sharp(F))^D.$$
\end{corollary}

Here is a trivial but important consequence: any $\eta \in (M(F)/M^\sharp(F))^D$ is invariant under the $W(M)$-action.

\subsection{Relation to intertwining operators}
Let $M$ be a Levi subgroup of $G$. First of all, observe that there is a natural decomposition
$$ \mathfrak{a}^*_M = \mathfrak{a}^*_{M^\sharp} \oplus \mathfrak{b}^* $$
where
$$ \mathfrak{b}^* := \mathcal{X}(M/M^\sharp) \otimes_\Z \R \hookrightarrow \mathfrak{a}^*_G .$$
Henceforth, we shall identify $\mathfrak{a}^*_{M^\sharp}$ as a vector subspace of $\mathfrak{a}^*_M$. Idem for their complexifications. This is compatible with restrictions in the following sense
$$ (\sigma_\lambda)|_{M^\sharp} = (\sigma|_{M^\sharp})_\lambda, \quad \sigma \in \Pi(M), \; \lambda \in \mathfrak{a}^*_{M^\sharp, \C}. $$

\begin{lemma}\label{prop:compatible-intertwinner}
  Let $\sigma \in \Pi(M)$, $P, Q \in \mathcal{P}(M)$. For $\lambda \in \mathfrak{a}^*_{M^\sharp, \C}$ in general position, the following diagram is commutative
  $$\xymatrix{
    I^G_P(\sigma_\lambda)|_{G^\sharp} \ar[d]_{\simeq} \ar[rrr]^{J_{Q|P}(\sigma_\lambda)} & & & I^G_Q(\sigma_\lambda)|_{G^\sharp} \ar[d]^{\simeq} \\
    I^{G^\sharp}_{P^\sharp}(\sigma_\lambda|_{M^\sharp}) \ar[rrr]_{J_{Q^\sharp|P^\sharp}(\sigma_\lambda|_{M^\sharp})} & & & I^{G^\sharp}_{Q^\sharp}(\sigma_\lambda|_{M^\sharp})
  }$$
  where the vertical isomorphisms are those defined in Lemma \ref{prop:res-induction}.
\end{lemma}
\begin{proof}
  It suffices to check this for $\Re(\lambda)$ in the cone of absolute convergence of the integrals \eqref{eqn:J-int} defining $J_{Q|P}$ and $J_{Q^\sharp|P^\sharp}$. The commutativity then follows from \eqref{eqn:J-int} and the definition of the isomorphism in Lemma \ref{prop:res-induction}.
\end{proof}

\begin{proposition}\label{prop:mu-compatibility}
  Let $\sigma \in \Pi_{2,\mathrm{temp}}(M)$, $\sigma^\sharp \in \Pi_{2,\mathrm{temp}}(M^\sharp)$ such that $\sigma^\sharp \hookrightarrow \sigma|_{M^\sharp}$. Then for all $\lambda \in \mathfrak{a}^*_{M^\sharp,\C}$, we have $\mu(\sigma_\lambda)=\mu(\sigma^\sharp_\lambda)$; more precisely,
  $$ \mu_\alpha(\sigma_\lambda) = \mu_\alpha(\sigma^\sharp_\lambda), \quad \forall \alpha \in \Sigma^{\mathrm{red}}_P = \Sigma^{\mathrm{red}}_{P^\sharp}, \; P \in \mathcal{P}(M). $$
\end{proposition}
\begin{proof}
  In view of our choice of measures on unipotent radicals, the identities of $\mu$-functions follow from \eqref{eqn:j-function} and Lemma \ref{prop:compatible-intertwinner}.
\end{proof}

\begin{proposition}\label{prop:mu-invariance}
  Let $\sigma \in \Pi(M)$ and $\eta \in (G(F)/G^\sharp(F))^D$. Then we have
  $$ j(\sigma) = j(\eta\sigma). $$
  In particular, for $\sigma \in \Pi_{2,\text{temp}}(M)$, we have $\mu(\sigma)=\mu(\eta\sigma)$.
\end{proposition}
\begin{proof}
  In view of the definition of $j$-function \eqref{eqn:j-function}, it suffices to observe that for all $P,Q \in \mathcal{P}(M)$ and $\lambda \in \mathfrak{a}^*_{M,\C}$ in general position, the following diagram commutes
  $$\xymatrix{
    \eta I^G_P(\sigma_\lambda) \ar[rr]^{J_{Q|P}(\sigma_\lambda)} \ar[d]_{\simeq} & & \eta I^G_Q(\sigma_\lambda) \ar[d]^{\simeq} \\
    I^G_P(\eta\sigma_\lambda) \ar[rr]_{J_{Q|P}(\eta\sigma_\lambda)} & & I^G_Q(\eta\sigma_\lambda)
  }$$
  where the vertical arrows are given by $\varphi(\cdot) \mapsto \eta(\cdot)\varphi(\cdot)$; note that we used the natural identification $\Hom_G(\pi_1, \pi_2) = \Hom_G(\eta\pi_1, \eta\pi_2)$ for all $\pi_1, \pi_2 \in \Pi(G)$. Indeed, the commutativity can be seen from the definition \eqref{eqn:J-int} of $J_{Q|P}(\cdot)$ when $\Re(\lambda)$ lies in the cone of absolute convergence.
\end{proof}

\begin{theorem}\label{prop:normalizing-invariance}
  One can choose a family of normalizing factors for $G$ such that
  $$ r_{Q|P}(\sigma_\lambda) = r_{Q|P}(\eta\sigma_\lambda) $$
  for all $(M, \sigma)$, $P,Q \in \mathcal{P}(M)$ and $\eta \in (G(F)/G^\sharp(F))^D$ which is unitary. Given such a family of normalizing factors, one can define normalizing factors $r_{Q^\sharp|P^\sharp}(\sigma^\sharp_\lambda)$ for those $\sigma^\sharp$ such that $\omega_{\sigma^\sharp}$ is unitary, by setting
  \begin{gather}\label{eqn:r-equality}
    r_{Q^\sharp|P^\sharp}(\sigma^\sharp_\lambda) := r_{Q|P}(\sigma_\lambda), \quad \lambda \in \mathfrak{a}^*_{M^\sharp,\C}
  \end{gather}
  where $\sigma \in \Pi(M)$ is as in Proposition \ref{prop:lifting} with $\omega_\sigma$ unitary.

  Moreover, let $\sigma$, $\sigma^\sharp$ be as above and $\iota: \sigma^\sharp \hookrightarrow \sigma|_{M^\sharp}$ be an embedding. Let $P,Q \in \mathcal{P}(M)$, $w \in W(M)$ with a representative $\tilde{w} \in G^\sharp(F)$. The following diagrams of $G^\sharp(F)$-representations are commutative
  $$\xymatrix{
    I^G_P(\sigma_\lambda) \ar[rr]^{R_{Q|P}(\sigma_\lambda)} & & I^G_Q(\sigma_\lambda) \\
    I^{G^\sharp}_{P^\sharp}(\sigma^\sharp_\lambda) \ar@{^{(}->}[u] \ar[rr]_{R_{Q^\sharp|P^\sharp}(\sigma^\sharp_\lambda)} & & I^{G^\sharp}_{Q^\sharp}(\sigma^\sharp_\lambda) \ar@{^{(}->}[u]
  } \qquad
  \xymatrix{
    I^G_P(\sigma_\lambda) \ar[rr]^{r_P(\tilde{w}, \sigma_\lambda)} & & I^G_P(\tilde{w}(\sigma_\lambda)) \\
    I^{G^\sharp}_{P^\sharp}(\sigma^\sharp_\lambda) \ar[rr]_{r_{P^\sharp}(\tilde{w}, \sigma^\sharp_\lambda)} \ar@{^{(}->}[u] & & I^{G^\sharp}_{P^\sharp}(\tilde{w}(\sigma^\sharp_\lambda)) \ar@{^{(}->}[u]
  }
  $$
  for $\lambda \in \mathfrak{a}^*_{M,\C}$ in general position, where the vertical arrows are given by
  $$ I^{G^\sharp}_{P^\sharp}(\sigma^\sharp_\lambda) \xrightarrow{I^{G^\sharp}_{P^\sharp}(\iota)} I^{G^\sharp}_{P^\sharp}(\sigma_\lambda|_{M^\sharp}) \xrightarrow{\sim} I^G_P(\sigma_\lambda)|_{G^\sharp}. $$
\end{theorem}

Observe that the asserted invariance under $\eta$-twist is satisfied by Langlands' conjectural family of normalizing factors in Remark \ref{rem:normalization}, since they are defined in terms of local factors through the adjoint representation of $\Lgrp{M}$ on the Lie algebra of $\hat{G}$.

\begin{proof}
  We will show that $r_{Q|P}(\sigma_\lambda) = r_{Q|P}(\eta\sigma_\lambda)$ by reviewing the construction in Remark \ref{rem:construction-normalization}. More precisely, we will start from the square-integrable case and show that the equality is preserved throughout the inductive construction.

  To begin with, suppose that $M$ is maximal proper and $\sigma \in \Pi_{2,\text{temp}}(M)$. Suppose that $r_{Q|P}(\sigma_\lambda)$ is chosen so that $\mathbf{R}_1$, $\mathbf{R}_2$, $\mathbf{R}_3$, $\mathbf{R}_6$, $\mathbf{R}_7$ are satisfied. Put $r_{Q|P}(\eta\sigma_\lambda) := r_{Q|P}(\sigma_\lambda)$, then all the conditions above except $\mathbf{R}_1$ are trivially satisfied for $\eta\sigma$. As for $\mathbf{R}_1$, all what we need to check is that when $Q=\bar{P}$,
  $$ r_{P|Q}(\sigma_\lambda) r_{Q|P}(\sigma_\lambda) = j(\eta\sigma_\lambda), \quad \lambda \in \mathfrak{a}^*_{M,\C}. $$
  By Proposition \ref{prop:mu-invariance}, the right hand side is equal to $j(\sigma_\lambda)$, hence the equality holds. This completes the case of $\sigma \in \Pi_{2,\text{temp}}(M)$.

  Suppose now $\sigma \in \Pi_{\text{temp}}(M)$. By the classification of tempered representations, we may write $\sigma \hookrightarrow I^M_R(\tau)$ for some parabolic subgroup $R=M_R U_R$ of $M$ and $\tau \in \Pi_{2,\text{temp}}(M_R)$. Twisting everything by $\eta$, we obtain $\eta\sigma \hookrightarrow I^M_R(\eta\tau)$ in the classification of tempered representations. We have $r_{Q|P}(\sigma) = r_{Q(R)|P(R)}(\tau) = r_{Q(R)|P(R)}(\eta\tau)$ by the previous case; on the other hand, the inductive construction of normalizing factors says that $r_{Q|P}(\eta\sigma)=r_{Q(R)|P(R)}(\eta\tau)$, hence $r_{Q|P}(\eta\sigma)=r_{Q|P}(\sigma)$.

  The case of general $\sigma$ is similar. We may write $\sigma$ as the Langlands quotient $I^M_R(\tau_\mu) \twoheadrightarrow \sigma$ where $R=M_R U_R$ is as before, $\tau \in \Pi_\text{temp}(M_R)$ and $\Re\angles{\mu,\alpha^\vee} > 0$ for all $\alpha \in \Delta^M_R$. Twisting everything by $\eta$, we have $I^M_R(\eta\tau_\mu) \twoheadrightarrow \eta\sigma$, which is still a Langlands quotient. The inductive construction of normalizing factors says that $r_{Q|P}(\sigma) = r_{Q(R)|P(R)}(\tau_{\mu+\lambda})$. Repeating the arguments for the previous case, it follows that $r_{Q|P}(\sigma) = r_{Q|P}(\eta\sigma)$.

  Now we can check that 
  $$ r_{Q^\sharp|P^\sharp}(\sigma^\sharp_\lambda) := r_{Q|P}(\sigma_\lambda) $$
  is well-defined. Recall that $\omega_{\sigma^\sharp}$ and $\omega_\sigma$ are assumed to be unitary. If $\sigma'$ is another choice such that $\sigma^\sharp \hookrightarrow \sigma|_{M^\sharp}$ and $\omega_{\sigma'}$ is unitary, then there exists $\eta$ such that $\sigma \simeq \eta\sigma'$. This would imply that $\eta|_{Z_G(F)}$ is unitary, hence so is $\eta$ itself. Therefore $r_{Q|P}(\eta\sigma_\lambda)=r_{Q|P}(\sigma_\lambda)$.

  Finally, the commutativity of the diagram results from Lemma \ref{prop:compatible-intertwinner} and \eqref{eqn:r-equality}.
\end{proof}

\subsection{Relation to $R$-groups}\label{sec:rel-R}
In this subsection, we will fix
\begin{itemize}
  \item a parabolic subgroup $P=MU$ of $G$;
  \item and the corresponding parabolic subgroup $P^\sharp := P \cap G^\sharp = M^\sharp U$ of $G^\sharp$;
  \item $\sigma^\sharp \in \Pi_{2,\text{temp}}(M^\sharp)$;
  \item $\sigma \in \Pi_{2,\text{temp}}(M)$ such that $\sigma^\sharp \hookrightarrow \sigma|_{M^\sharp}$.
\end{itemize}
Given $\sigma^\sharp$, the existence of such a $\sigma$ is guaranteed by Propositions \ref{prop:lifting} and \ref{prop:heredity}.

\begin{lemma}[cf. {\cite[Lemma 2.3]{Go06}}]\label{prop:W^0-equality}
  Under the identification $W(M)=W(M^\sharp)$, we have $W^0_\sigma = W^0_{\sigma^\sharp}$.
\end{lemma}
\begin{proof}
  Since both sides are generated by root reflections, it suffices to fix $\alpha \in \Sigma^{\text{red}}_P = \Sigma^{\text{red}}_{P^\sharp}$ and show that $s_\alpha \in W^0_\sigma$ if and only if $s_\alpha \in W^0_{\sigma^\sharp}$, where $s_\alpha$ denotes the root reflection with respect to $\alpha$.

  By \cite[Proposition IV.2.2]{Wa03}, $\mu_\alpha(\sigma)=0$ implies $s_\alpha \sigma \simeq \sigma$. Idem for $\sigma^\sharp$ instead of $\sigma$. According to the description of $W^0_\sigma$ (resp. $W^0_{\sigma^\sharp}$) in terms of $\mu$-functions, we obtain
  $$ \mu_\alpha(\sigma)=0 \Leftrightarrow s_\alpha \in W^0_\sigma \quad (\text{resp. } \mu_\alpha(\sigma^\sharp)=0 \Leftrightarrow s_\alpha \in W^0_{\sigma^\sharp}). $$

  On the other hand, Proposition \ref{prop:mu-compatibility} implies $\mu_\alpha(\sigma)=\mu_\alpha(\sigma^\sharp)$. The assertion follows immediately.
\end{proof}

\begin{definition}\label{def:L}
  Set
  \begin{align*}
    L(\sigma) & := \left\{ \eta \in (M(F)/M^\sharp(F))^D : \exists w \in W(M), \; w\sigma \simeq \eta\sigma \right\}, \\
    L(\sigma^\sharp) & := \left\{ \eta \in (M(F)/M^\sharp(F))^D : \exists w \in W(M), \; w\sigma \simeq \eta\sigma, \; w\sigma^\sharp \simeq \sigma^\sharp \right\}.
  \end{align*}
  They are subgroups of $(M(F)/M^\sharp(F))^D$. Indeed, let $\eta, \eta' \in L(\sigma)$ and $w, w' \in W(M)$ such that $\eta\sigma \simeq w\sigma$, $\eta'\sigma \simeq w'\sigma$. Then one has
  \begin{gather}\label{eqn:group-law-L}
    \eta'\eta\sigma \simeq \eta' w\sigma = w \eta'\sigma \simeq ww'\sigma.
  \end{gather}
  Hence $\eta\eta' \in L(\sigma)$. The case of $L(\sigma^\sharp)$ is similar. Note that $X^M(\sigma) \subset L(\sigma^\sharp) \subset L(\sigma)$.

  There is an obvious counterpart for the Weyl group, namely
  $$ \bar{W}_\sigma := \left\{ w \in W(M) : \exists \eta \in (M(F)/M^\sharp(F))^D, \; w\sigma \simeq \eta\sigma \right\}. $$
  It is clear that $\bar{W}_\sigma \supset W_\sigma$. On the other hand, Proposition \ref{prop:res-disjoint} implies that $\bar{W}_\sigma \supset W_{\sigma^\sharp}$.
\end{definition}

The following result is clear in view of the preceding definitions.
\begin{lemma}\label{prop:barGamma}
  There is a homomorphism given by
  \begin{align*}
    \bar{\Gamma}: \bar{W}_\sigma & \longrightarrow L(\sigma)/X^M(\sigma) \\
    w & \longmapsto \text{ the } \left[ \eta \text{ mod } X^M(\sigma) \right] \text{ such that } w\sigma \simeq \eta\sigma
  \end{align*}
  which satisfies
  \begin{enumerate}
    \item $\bar{\Gamma}$ is surjective;
    \item $\Ker(\bar{\Gamma}) = W_\sigma$;
    \item the preimage of $L(\sigma^\sharp)/X^M(\sigma)$ is $W_{\sigma^\sharp} W_\sigma$.
  \end{enumerate}
\end{lemma}

\begin{definition}
  Let
  $$ \Gamma: \bar{W}_\sigma/W_\sigma \rightiso L(\sigma)/X^M(\sigma) $$
  be the isomorphism obtained from $\bar{\Gamma}$ in the previous Lemma.
\end{definition}

\begin{proposition}[{Goldberg \cite[Proposition 3.2]{Go06}}]\label{prop:Goldberg}
  Set
  $$ R_\sigma[\sigma^\sharp] := (W_\sigma \cap W_{\sigma^\sharp})/W^0_{\sigma^\sharp} $$
  which is legitimate by Lemma \ref{prop:W^0-equality}. It is a subgroup of $R_{\sigma^\sharp}$.
  \begin{enumerate}
    \item The homomorphism $\bar{\Gamma}$ induces an isomorphism
      $$ \Gamma: R_{\sigma^\sharp}/R_\sigma[\sigma^\sharp] \rightiso L(\sigma^\sharp)/X^M(\sigma). $$
    \item If $R_\sigma = \{1\}$, or equivalently if $I^G_P(\sigma)$ is irreducible, then $\Gamma$ induces an isomorphism $R_{\sigma^\sharp} \rightiso L(\sigma^\sharp)/X^M(\sigma)$. Consequently, $R_{\sigma^\sharp}$ is abelian in this case.
  \end{enumerate}
\end{proposition}
\begin{proof}
  Lemma \ref{prop:barGamma} gives an isomorphism
  $$ W_{\sigma^\sharp}/(W_\sigma \cap W_{\sigma^\sharp}) \rightiso L(\sigma^\sharp)/X^M(\sigma) $$
  that can be viewed as a restriction of $\Gamma$. By Lemma \ref{prop:W^0-equality}, we can take the quotients by $W^0_\sigma=W^0_{\sigma^\sharp}$ on the left hand side. The first assertion follows immediately.

  For the second assertion, it suffices to note that $R_\sigma[\sigma^\sharp]$ embeds into $R_\sigma$ as well, since $W^0_\sigma=W^0_{\sigma^\sharp}$.
\end{proof}

\subsection{L-parameters}\label{sec:res-L}
Let $G$ be a connected reductive $F$-group equipped with a quasisplit inner twist
$$ \psi: G \times_F \bar{F} \to G^* \times_F \bar{F}.$$

We identify $\hat{G}$ with $\widehat{G^*}$, thus $\Lgrp{G}=\Lgrp{G^*}$. The reader should recall that the definition of the complex reductive group $\widehat{G^*}$ and the $\Gamma_F$-action thereof depend on the choice of a $\Gamma_F$-stable splitting $(B^*, T^*, (E_\alpha)_{\alpha \in \Delta(B^*,T^*)})$ (also known as an $F$-splitting) of $G^*(\bar{F})$ (see \cite[\S 1]{Ko84}). These choices permit to define a correspondence $M^* \leftrightarrow \Lgrp{M^*}$ between the conjugacy classes of Levi subgroups of $G^*$ and their dual avatars inside $\Lgrp{G^*}$. Using the inner twist $\psi$, it also makes sense to say if a Levi subgroup $M^*$ of $G^*$ comes from $G$; this notion only depends on the conjugacy classes of Levi subgroups.

For any Levi subgroup $M$ of $G$, there is a canonical bijection between $W^G(M)$ and $W^{\hat{G}}(\hat{M})$ coming from the bijection between roots and coroots.

An L-parameter for $G^*$ is a homomorphism
$$ \phi: \WD_F \to \Lgrp{G^*} = \Lgrp{G} $$
such that
\begin{itemize}
  \item $\phi$ is an L-homomorphism, i.e. the composition of $\phi$ with the projection $\Lgrp{G} \to W_F$ equals $\WD_F \to W_F$;
  \item $\phi$ is continuous;
  \item the projection of $\Im(\phi)$ to $\hat{G}$ is formed of semisimple elements.
\end{itemize}

Two L-parameters $\phi_1$, $\phi_2$ are called equivalent, denoted by $\phi_1 \sim \phi_2$, if they are conjugate by $\hat{G}$. We say that $\phi$ is bounded if the projection of $\Im(\phi)$ to $\hat{G}$ is bounded (i.e.\! relatively compact); this property depends only on the equivalence class of $\phi$.

Given an L-parameter $\phi$, we define
$$ S_\phi := Z_{\hat{G}}(\Im(\phi)). $$
The connected component $S^0_\phi$ is a connected reductive subgroup of $\hat{G}$. We record the following basic properties.
\begin{enumerate}
  \item the Levi subgroups $\Lgrp{M^*} \subset \Lgrp{M}$ which contain $\Im(\phi)$ minimally are conjugate by $S^0_\phi$.
  \item Let $\Lgrp{M^*}$ be a Levi subgroup containing $\Im(\phi)$ minimally, then $Z_{\widehat{M^*}}^{\Gamma_F, 0}$ is a maximal torus of $S^0_\phi$.
\end{enumerate}
Indeed, these assertions follow from \cite[Proposition 3.6]{Bo79} and its proof, applied to the subgroup $\Im(\phi)$ of $\Lgrp{G}$.

So far, everything depends only on the quasisplit inner form $G^*$. We say that $\phi$ is $G$-relevant if $M^*_\phi$ corresponds to a Levi subgroup of $G$; in this case, we write $M^*_\phi = M_\phi$. Put

\begin{align*}
  \Phi(G) & := \{\phi : \WD_F \to \Lgrp{G}, \; \phi \text{ is a $G$-relevant L-parameter}\}/\sim, \\
  \Phi_\text{bdd}(G) & := \{\phi \in \Phi(G) : \phi \text{ is bounded} \}, \\
  \Phi_{2,\text{bdd}}(G) & := \{\phi \in \Phi_\text{bdd}(G) : M_\phi = G \}.
\end{align*}

Since the relevance condition is vacuous if $G=G^*$, we have $\Phi(G) \subset \Phi(G^*)$, etc.

Now let $G^\sharp$ be a subgroup of $G$ such that $G_\text{der} \subset G^\sharp \subset G$. We will study the lifting of L-parameters from $G^\sharp$ to $G$, which is in some sense dual to the restriction of representations. In what follows, the L-groups of $G$ and $G^\sharp$ will be defined using compatible choices of quasisplit inner twists and $F$-splittings.

There is a natural, $\Gamma_F$-equivariant central extension
$$ 1 \to \hat{Z}^\sharp \to \hat{G} \xrightarrow{\mathbf{pr}} \widehat{G^\sharp} \to 1 $$
which is dual to $G^\sharp \to G$; here $\hat{Z}^\sharp$ is the $\C$-torus dual to $G/G^\sharp$.

For $\phi \in \Phi(G)$, we shall set $\phi^\sharp := \mathbf{pr} \circ \phi \in \Phi(G^\sharp)$. When this equality holds, $\phi$ is called a lifting of $\phi^\sharp$.

\begin{theorem}\label{prop:lifting-parameter}
  For any $\phi^\sharp \in \Phi(G^\sharp)$, there exists a lifting $\phi \in \Phi(G)$ of $\phi^\sharp$ which is unique up to twists by $H^1_\mathrm{cont}(W_F, \hat{Z}^\sharp)$. If $\phi^\sharp \in \Phi_\mathrm{bdd}(G^\sharp)$ (resp. $\phi^\sharp \in \Phi_{2,\mathrm{bdd}}(G^\sharp)$), then $\phi$ can be chosen so that $\phi \in \Phi_\mathrm{bdd}(G)$ (resp. $\phi \in \Phi_{2,\mathrm{bdd}}(G)$).
\end{theorem}
Note that by local class field theory, $H^1_\mathrm{cont}(W_F, \hat{Z}^\sharp)$ parametrizes the continuous characters of $(G(F)/G^\sharp(F))^D$.

\begin{proof}
  The existential part is just \cite[Théorème 8.1]{Lab85} and the uniqueness follows easily. Assume that $\phi^\sharp \in \Phi_\mathrm{bdd}(G^\sharp)$ (resp. $\phi^\sharp \in \Phi_{2,\mathrm{bdd}}(G^\sharp)$) and let $\phi$ be any lifting of $\phi^\sharp$; we have to show that there exists a continuous $1$-cocycle $a: W_F \to \hat{Z}^\sharp$ such that the twisted L-parameter $a\phi$ is bounded (resp. bounded and satisfying $M_{a\phi}=M_\phi=G$).

  Note that there exists a central isogeny of connected reductive groups
  $$ G^\sharp \times C \to G $$
  given by multiplication, where $C$ is some subtorus of $Z^0_G$. Hence $C \to G/G^\sharp$ is also an isogeny. By duality, we obtain a $\Gamma_F$-equivariant central isogeny of connected reductive complex groups
  $$ \hat{G} \to \widehat{G^\sharp} \times \hat{C}. $$

  Let $\phi'$ be the composition of $\phi$ (projected to the $\hat{G}$ component) with the aforementioned central isogeny. Let us show that $\Im(\phi')$ is bounded upon twisting $\phi$. The first component of $\phi'$ is automatically bounded since $\phi^\sharp$ is. On the other hand, $\hat{C}$ is isogeneous to $\hat{Z}^\sharp$, therefore upon replacing $\phi$ by $a\phi$ for some suitable $1$-cocycle $a: W_F \to \hat{Z}^\sharp$, the second component can be made bounded. Hence $a\phi$ is a bounded L-parameter. 

  To finish the proof, it remains to observe that assuming $\phi^\sharp = \mathbf{pr} \circ \phi$, the preimage of $M^\sharp_{\phi^\sharp}$ in $\hat{G}$ is equal to $M_\phi$.
\end{proof}

Here we record a construction related to inner forms which will be required later. Recall that the inner forms of $G^*$ are parametrized by $H^1(F, G^*_\text{AD})$. Kottwitz \cite[\S 6]{Ko84} defined the ``abelianization'' map $\text{ab}^1: H^1(F, G^*_\text{AD}) \to (Z^{\Gamma_F}_{\hat{G}_\text{SC}})^D$ between pointed sets. Hence we can associate to $G$ a character $\chi_G$ of $Z^{\Gamma_F}_{\hat{G}_\text{SC}}$, namely by
\begin{equation}\label{eqn:chi-G}\begin{split}
  \{ \text{inner forms of } G^* \}  = H^1(F, G^*_\text{AD}) \xrightarrow{\text{ab}^1} (Z^{\Gamma_F}_{\hat{G}_\text{SC}})^D, \\
  G \longmapsto \left[ \chi_G : Z^{\Gamma_F}_{\hat{G}_\text{SC}} \to \C^\times \right].
\end{split}\end{equation}

\section{Restriction, continued}\label{sec:res-2}
This section is devoted to the study of restriction under parabolic induction. As before, we fix connected reductive $F$-groups $G$, $G^\sharp$ such that $G_\text{der} \subset G^\sharp \subset G$. We also fix a Levi subgroup $M$ of $G$ and $P \in \mathcal{P}(M)$. The bijection between Levi subgroups (resp. parabolic subgroups) $M \mapsto M^\sharp$ (resp. $P \mapsto P^\sharp$) is defined in \S\ref{sec:res-ind}.

The normalizing factors for $G$, $G^\sharp$ are chosen as in Theorem \ref{prop:normalizing-invariance} for the representations with unitary central character.

Let $\sigma \in \Pi(M)$. We shall make the following (rather restrictive) hypothesis on $\sigma$ throughout this section.
\begin{hypothesis}\label{hyp:irred}
  We assume that $\pi := I^G_P(\sigma)$ is irreducible.
\end{hypothesis}

\subsection{Embedding of central extensions}\label{sec:embedding}
\begin{proposition}\label{prop:S-embedding}
  Let $\sigma \in \Pi(M)$ and $\pi := I^G_P(\sigma) \in \Pi(G)$. 
  \begin{enumerate}
    \item Under the identification of Corollary \ref{prop:char-G-M}, we have $X^M(\sigma) \hookrightarrow X^G(\pi)$.
    \item Let $\omega \in X^M(\sigma)$, $I^M_\omega \in \Isom_M(\omega\sigma, \sigma)$, define the operator $I^G_\omega$ as the composition of
    \begin{align*}
        A_\omega: \omega I^G_P(\sigma) & \longrightarrow I^G_P(\omega\sigma) \\
        \varphi & \longmapsto \omega(\cdot)\varphi(\cdot)
    \end{align*}
    with $I^G_P(I^M_\omega): I^G_P(\omega\sigma) \rightiso I^G_P(\sigma)$. Then $I^G_\omega \in \Isom_G(\omega\pi, \pi)$.
    \item We have the following commutative diagram of groups with exact rows
    $$\xymatrix{
      1 \ar[r] & \C^\times \ar[r] & S^G(\pi) \ar[r] & X^G(\pi) \ar[r] & 1 \\
      1 \ar[r] & \C^\times \ar@{=}[u] \ar[r] & S^M(\sigma) \ar@{^{(}->}[u] \ar[r] & X^M(\sigma) \ar@{^{(}->}[u] \ar[r] & 1
    }$$
    where the arrow $S^M(\sigma) \to S^G(\pi)$ is the map $I^M_\omega \mapsto I^G_\omega$ defined above.
  \end{enumerate}
\end{proposition}
\begin{proof}
  It follows from the definition that $I^G_\omega \in S^G(\pi)$ for all $\omega \in X^M(\sigma)$, hence $X^M(\sigma) \subset X^G(\pi)$. On the other hand, $I^G_\omega$ is simply the map $\varphi \mapsto \omega(\cdot) I^M_\omega(\varphi(\cdot))$. It is clear that $I^M_\omega \mapsto I^G_\omega$ is a group homomorphism. The commutativity of the diagram is then clear.
\end{proof}

We denote by $\mathbf{K}_0(\Pi_-(S^M(\sigma)))$ the space of virtual characters of $S^M(\sigma)$ generated by the elements of $\Pi_-(S^M(\sigma))$. Similarly, $\mathbf{K}_0(\Pi_\sigma)$ denotes the space of virtual characters of $M^\sharp(F)$ generated by the elements of $\Pi_\sigma$. The bijection $\rho \mapsto \pi^\sharp_\rho$ in Theorem \ref{prop:S-decomp} extends to an isomorphism $\mathbf{K}_0(\Pi_-(S^M(\sigma))) \rightiso \mathbf{K}_0(\Pi_\sigma)$. Assuming $\pi = I^G_P(\sigma)$ irreducible, we have the analogous isomorphism $\mathbf{K}_0(\Pi_-(S^G(\pi))) \rightiso \mathbf{K}_0(\Pi_\pi)$, as well as the linear maps
\begin{align*}
  \text{Ind}^{S^G(\pi)}_{S^M(\sigma)}: & \mathbf{K}_0(\Pi_-(S^M(\sigma))) \longrightarrow \mathbf{K}_0(\Pi_-(S^G(\pi))), \\
  I^{G^\sharp}_{P^\sharp}: & \mathbf{K}_0(\Pi_\sigma) \longrightarrow \mathbf{K}_0(\Pi_\pi),
\end{align*}
given by the usual induction and normalized parabolic induction, respectively. Note that Lemma \ref{prop:res-induction} is invoked here.

\begin{proposition}\label{prop:K_0-diagram}
  The following diagram commutes.
  $$\xymatrix{
    \mathbf{K}_0(\Pi_-(S^G(\pi))) \ar[rr]^{\simeq} & & \mathbf{K}_0(\Pi_\pi) \\
    \mathbf{K}_0(\Pi_-(S^M(\sigma))) \ar[rr]_{\simeq} \ar[u]^{\mathrm{Ind}^{S^G(\pi)}_{S^M(\sigma)}} & & \mathbf{K}_0(\Pi_\sigma) \ar[u]_{I^{G^\sharp}_{P^\sharp}}
  }$$
\end{proposition}

To prove this, some harmonic analysis on the groups $S^M(\sigma)$, $S^G(\pi)$ is needed. These groups are infinite; nonetheless, the usual theory carries over as we are only concerned about the representations in $\Pi_-(S^G(\pi))$, $\Pi_-(S^M(\sigma))$ or their contragredients. The worried reader may reduce $S^G(\pi) \to X^G(\pi)$ (resp. $S^M(\sigma) \to X^M(\sigma)$) to a central extension by $\mu_m := \{z \in \C^\times : z^m = 1 \}$ for some $m \in \Z$, which is always possible.

\begin{proof}
  Let $\sigma^\sharp \in \Pi_\sigma$ and $\rho \in \Pi_-(S^M(\sigma))$ be the corresponding element. Define
  \begin{align*}
    \tau & := \text{Ind}^{S^G(\pi)}_{S^M(\sigma)}(\rho) \in \mathbf{K}_0(\Pi_-(S^G(\pi))), \\
    \pi^\sharp & := I^{G^\sharp}_{P^\sharp}(\sigma^\sharp) \in \mathbf{K}_0(\Pi_\pi).
  \end{align*}

  We have to show that $\tau$ corresponds to $\pi^\sharp$. To begin with, set
  $$ \sigma[I^M_\omega] := \sigma(\cdot) \circ I^M_\omega : M(F) \to \Aut_\C(V_\sigma), \quad I^M_\omega \in S^M(\sigma) $$
  where $V_\sigma$ is the underlying vector space of $\sigma$. Then $(\sigma[I^M_\omega],\sigma)$ is a smooth $\omega$-representation of $M(F)$, that is,
  $$ \sigma[I^M_\omega](xy) = \omega(y) \sigma[I^M_\omega](x)\sigma(y), \quad x,y \in M(F). $$
  This notion appears in the study of automorphic induction, and more generally it fits into the formalism of twisted endoscopy. Cf. \cite[\S 0.4]{L10}. It is easy to see that $\Theta_\sigma[I^M_\sigma] := \Tr\sigma[I^M_\sigma]$ is well-defined as a distribution on $M(F)$. We may restrict $\sigma[I^M_\omega]$ to $M^\sharp(F)$; by abuse of notations, the corresponding distribution, which is also well-defined by Proposition \ref{prop:lifting}, is again denoted by $\Theta_\sigma[I^M_\sigma]$. The same definition applies to $\pi$.

  Theorem \ref{prop:S-decomp} implies the following identity of distributions on $M^\sharp(F)$
  \begin{gather}\label{eqn:sigma-sharp-inversion}
    \Theta_{\sigma^\sharp} = \frac{1}{|X^M(\sigma)|} \sum_{\omega \in X^M(\sigma)} \Tr(\rho^{\vee})\left( I^M_\omega \right) \cdot \Theta_\sigma[I^M_\omega]
  \end{gather}
  where $\rho^\vee$ is the contragredient of $\rho$ and $I^M_\omega \in S^M(\sigma)$ is any preimage $\omega$; the summand does not depend on the choice of $I^M_\omega$.

  Define $Z^M(\sigma)$ to be the subgroup of elements $\omega \in X^M(\sigma)$ such that every preimage of $\omega$ in $S^M(\sigma)$ is central. Define $Z^G(\pi)$ similarly. The sum in \eqref{eqn:sigma-sharp-inversion} can be taken over $Z^M(\sigma)$, since $\rho|_{\C^\times} = \identity$ implies that $\Tr(\rho^{\vee})$ is zero outside the center.

  Let $\pi^\sharp_1 \in \mathbf{K}_0(\Pi_\pi)$ be the character corresponding to $\tau$. By the same reasoning, there is an identity of distributions on $G^\sharp(F)$
  \begin{gather}\label{eqn:pi-sharp_1-inversion}
    \Theta_{\pi^\sharp_1} = \frac{1}{|X^G(\pi)|} \sum_{\eta \in Z^G(\pi)} \Tr(\tau^{\vee})\left( I^G_\eta \right) \cdot \Theta_\pi[I^G_\eta]
  \end{gather}
  where $I^G_\eta \in S^G(\pi)$ is any preimage of $\eta$, as before. It remains to show $\Theta_{\pi^\sharp_1}(f^\sharp) = \Theta_{\pi^\sharp}(f^\sharp)$ for every $f^\sharp \in C^\infty_c(G^\sharp(F))$.

  Choose a special maximal compact open subgroup $K \subset G(F)$ in good position relative to $M$, and set $K^\sharp := K \cap G^\sharp(F)$. Equip $K$ and $K^\sharp$ with appropriate Haar measures that are compatible with the Iwasawa decomposition (see \cite[I.1]{Wa03}). The parabolic descent of characters implies
  \begin{gather}\label{eqn:pi-sharp-inversion}
    \Theta_{\pi^\sharp}(f^\sharp) = \Theta_{\sigma^\sharp}(f^\sharp_{P^\sharp})
    = \frac{1}{|X^M(\sigma)|} \sum_{\omega \in Z^M(\sigma)} \Tr(\rho^{\vee})\left( I^M_\omega \right) \cdot \Theta_\sigma[I^M_\sigma](f^\sharp_{P^\sharp}),
  \end{gather}
  where
  $$ f^\sharp_{P^\sharp}(m) = \delta^{\frac{1}{2}}_P(m) \iint_{U(F) \times K^\sharp} f^\sharp(k^{-1}muk) \dd u \dd k, \quad m \in M^\sharp(F). $$

  Since $\tau = \text{Ind}^{S^G(\pi)}_{S^M(\sigma)}(\rho)$, we have
  $$ \Tr(\tau^\vee)\left( I^G_\eta \right) = \begin{cases}
    \dfrac{1}{|X^M(\sigma)|} \sum_{\xi \in X^G(\pi)} \Tr(\rho^\vee) \left((I^G_\xi)^{-1} I^M_\eta I^G_\xi \right), & \text{if } \eta \in X^M(\sigma), \; I^M_\eta \mapsto I^G_\eta, \\
    0, & \text{otherwise},
  \end{cases}$$
  where $I^G_\xi \in S^G(\pi)$ is any preimage of $\xi$. Cf. \cite[Proposition 20]{Se67}. This may be rewritten as
  $$ \Tr(\tau^\vee)\left( I^G_\eta \right) = \begin{cases}
    \dfrac{|X^G(\pi)|}{|X^M(\sigma)|} \Tr(\rho^\vee) \left(I^M_\eta \right), & \text{if } \eta \in X^M(\sigma) \cap Z^G(\pi), \; I^M_\eta \mapsto I^G_\eta, \\
    0, & \text{otherwise}.
  \end{cases}$$

  We claim that $\Theta_\pi[I^G_\omega](f^\sharp) = \Theta_\sigma[I^M_\omega](f^\sharp_{P^\sharp})$ if $I^M_\sigma \mapsto I^G_\sigma \in S^G(\pi)$. First of all, note that $\Theta_\pi[I^G_\omega]$ is the normalized parabolic induction of $\Theta_\sigma[I^M_\omega]$ in the setting of $\omega$-representations \cite[\S 1.7 and \S 3.8]{L10}. Hence we have the parabolic descent of $\omega$-characters \cite[Théorème 3.8.2]{L10}, namely
  $$ \Theta_\pi[I^G_\omega](f) = \Theta_\sigma[I^M_\omega](f_{P,\omega}), \quad f \in C^\infty_c(G(F)) $$
  where
  $$ f_{P,\omega}(m) = \delta^{\frac{1}{2}}_P(m) \iint_{U(F) \times K} \omega(k) f(k^{-1}muk) \dd u \dd k, \quad m \in M(F). $$
  To prove the claim, let us sketch how to ``restrict'' the $\omega$-character relation above to $G^\sharp(F)$. There exists a compact open subgroup $C \subset Z_G(F)$, verifying
  \begin{enumerate}
    \item $C \cap G^\sharp(F) = \{1\}$;
    \item $C \subset K$;
    \item $\omega$ and $\omega_\sigma$ are trivial on $C$;
    \item the multiplication maps $C \times G^\sharp(F) \hookrightarrow G(F)$ and $C \times M^\sharp(F) \hookrightarrow M(F)$ are submersive.
  \end{enumerate}
  Define $\mathbbm{1}_C$ to be the constant function $1$ on $C$. Choose the unique Haar measure on $C$ such that the submersions above preserve measures locally. Given $f^\sharp \in C^\infty_c(G^\sharp(F))$, we set $f = \text{vol}(C)^{-1} \mathbbm{1}_C \otimes f^\sharp$ on $C \times G^\sharp(F)$, and zero elsewhere. For such $f$, by inspecting the proof in \cite[Proposition 1.8.1]{L10}, we may redefine $f_{P,\omega}$ by taking the double integral of $f(k^{-1}muk)$ over $U(F) \times K^\sharp$, so that $f_{P,\omega} = \text{vol}(C)^{-1} \mathbbm{1}_C \otimes f^\sharp_{P^\sharp}$ on $C \times M^\sharp(F)$ and zero elsewhere. Therefore
  \begin{align*}
    \Theta_\pi[I^G_\omega](f) & = \Theta_\pi[I^G_\omega](f^\sharp), \\
    \Theta_\sigma[I^M_\sigma](f_{P,\omega}) & = \Theta_\sigma[I^M_\sigma](f^\sharp_{P^\sharp}).
  \end{align*}
  Hence our claim follows.

  All in all, \eqref{eqn:pi-sharp_1-inversion} becomes
  $$ \Theta_{\pi^\sharp_1}(f^\sharp) =  \frac{1}{|X^M(\sigma)|} \sum_{\omega \in X^M(\sigma) \cap Z^G(\pi)} \Tr(\rho^\vee)(I^M_\omega) \cdot \Theta_\sigma[I^M_\omega](f^\sharp_{P^\sharp}) $$
  where $I^M_\sigma \in S^M(\sigma)$ is any preimage of $\omega$ and $I^M_\sigma \mapsto I^G_\sigma \in S^G(\pi)$. In comparison with \eqref{eqn:pi-sharp-inversion}, it suffices to show that $\Theta_\sigma[I^M_\omega](f^\sharp_{P^\sharp}) = 0$ if $\omega \in Z^M(\sigma)$ but $\omega \notin Z^G(\pi)$. Indeed, for $I \in S^G(\pi)$, we have
  \begin{align*}
    \Theta_\sigma[I^M_\omega](f^\sharp_{P^\sharp}) & = \Theta_\pi[I^G_\omega](f^\sharp) = \Tr\mathfrak{S}(I^G_\omega, f^\sharp) \\
    & = \Tr\mathfrak{S}(I^{-1} I^G_\omega I, f^\sharp)
  \end{align*}
  since $\mathfrak{S}$ is a representation of $S^G(\pi) \times G^\sharp(F)$. Since $I$ is arbitrary and $\mathfrak{S}|_{\C^\times \times \{1\}} = \identity$, we conclude that $\Theta_\sigma[I^M_\omega](f^\sharp_{P^\sharp}) \neq 0$ only if $\omega \in Z^G(\pi)$.
\end{proof}

\subsection{Description of $R$-groups}\label{sec:desc-R}
Let $w \in W(M)$ with a chosen representative $\tilde{w} \in G^\sharp(F)$. Recall the operator $r_P(\tilde{w}, \sigma): I^G_P(\sigma) \to I^G_P(\tilde{w}\sigma)$ defined in \eqref{eqn:r_P}, which is the composition of $R_{w^{-1}Pw|P}(\sigma): I^G_P(\sigma) \to I^G_{w^{-1}Pw|P}(\sigma)$ with the isomorphism
\begin{align*}
  \ell(\tilde{w}): I^G_{w^{-1}Pw}(\sigma) & \longrightarrow I^G_P(\tilde{w}\sigma) \\
  \varphi(\cdot) & \longmapsto \varphi(\tilde{w}^{-1} \cdot).
\end{align*}

For $\eta \in (G(F)/G^\sharp(F))^D$, recall the isomorphism $A_\eta$ defined as
\begin{align*}
  \eta I^G_P(\sigma) & \longrightarrow I^G_P(\eta\sigma) \\
  \varphi(\cdot) & \longmapsto \eta(\cdot)\varphi(\cdot).
\end{align*}

Note that the representations $\sigma$, $\eta\sigma$, $\tilde{w}\sigma$ and $\eta\tilde{w}\sigma$ share the same underlying vector space $V_\sigma$. As usual, we will compose the operators above after appropriate twists by $\eta$ or $\tilde{w}$. For example, given $\eta, \eta'$, we may define $A_\eta A_{\eta'}$ which is equal to $A_{\eta\eta'}: \eta\eta' I^G_P(\sigma) \to I^G_P(\eta\eta' \sigma)$.

\begin{proposition}\label{prop:L-X^G}
  Let $L(\sigma) \subset (M(F)/M^\sharp(F))^D$ be the subgroup defined in Definition \ref{def:L}. Upon identifying $M(F)/M^\sharp(F)$ and $G(F)/G^\sharp(F)$, we have
  $$ L(\sigma) \subset X^G(\pi).$$

  If $\sigma \in \Pi_{2,\mathrm{temp}}(M)$, then equality holds.
\end{proposition}
\begin{proof}
  Let $\eta \in L(\sigma)$. By definition, there exists $w \in W(M)$ with a representative $\tilde{w} \in G^\sharp(F)$ such that $\eta\sigma \simeq \tilde{w}\sigma$. Hence $\eta\pi \simeq I^G_P(\eta\sigma) \simeq I^G_P(\tilde{w}\sigma)$. There is also an isomorphism $r_P(\tilde{w}^{-1}, \tilde{w}\sigma): I^G_P(\tilde{w}\sigma) \rightiso I^G_P(\sigma)$. Hence $\eta\pi \simeq \pi$.

  Assume $\sigma \in \Pi_{2,\mathrm{temp}}(M)$ and let $\eta \in X^G(\pi)$, then $I^G_P(\eta\sigma) \simeq I^G_P(\sigma)$. By \cite[Proposition III.4.1]{Wa03}, there exists $w \in W(M)$ with $w\sigma \simeq \eta\sigma$. 
\end{proof}

Henceforth we take $w \in \bar{W}_\sigma$. Take any $\eta \in L(\sigma)$ whose class modulo $X^M(\sigma)$ equals $\bar{\Gamma}(w)$ (see Proposition \ref{prop:barGamma}). Then $\eta$ is of finite order by Proposition \ref{prop:L-X^G}, in particular $\eta$ is unitary. For any isomorphism
$$ i: \eta\sigma \rightiso \tilde{w}\sigma, $$
we deduce an isomorphism
\begin{align*}
  \Ad(i): S^M(\eta\sigma) & \longrightarrow S^M(\tilde{w}\sigma) \\
  I & \longmapsto iIi^{-1} =: \Ad(i)I.
\end{align*}

By the obvious identifications $S^M(\sigma) = S^M(\eta\sigma) = S^M(\tilde{w}\sigma)$ (without using $i$), one can view $\Ad(i)$ as an automorphism of $S^M(\sigma)$. It leaves $\C^\times$ intact and covers the identity map $X^M(\eta\sigma) \rightiso X^M(\tilde{w}\sigma)$, hence induces a bijection
\begin{align*}
  \Pi_-(S^M(\sigma)) & \longrightarrow \Pi_-(S^M(\sigma)) \\
  \rho & \longmapsto \rho \circ \Ad(i).
\end{align*}
We shall write
$$ \tilde{w}\rho := \rho \circ \Ad(i).$$
The puzzling notation will be justified by Lemma \ref{prop:rho-sigma-correspondence}.

Henceforth, we adopt the following convention: for $I^G_\eta \in S^G(\pi)$, we regard $\Ad(I^G_\eta)$ as an automorphism of $S^M(\sigma)$ via the embedding $S^M(\sigma) \hookrightarrow S^G(\pi)$ provided by Proposition \ref{prop:S-embedding}.

\begin{lemma}\label{prop:w-eta}
  Let $\eta$, $\tilde{w}$ and $i: \eta\sigma \rightiso \tilde{w}\sigma$ be as above. As automorphisms of $S^M(\sigma)$, we have
  $$ \Ad(i) = \Ad(I^G_\eta) $$
  for every $I^G_\eta \in S^G(\pi)$ in the preimage of $\eta$.
\end{lemma}
\begin{proof}
  Given $\eta$, the assertion is independent of the choice of $I^G_\eta$. Let us consider the specific choice as follows
  $$ I^G_\eta := r_P(\tilde{w}^{-1}, \tilde{w}\sigma) \circ I^G_P(i) \circ A_\eta: \eta I^G_P(\sigma) \to I^G_P(\sigma). $$

  By the definition of the embedding $S^M(\sigma) \hookrightarrow S^G(\pi)$ and that of $A_\eta$, one sees that $\Ad(A_\eta)$ induces the identity map $S^M(\sigma) \rightiso S^M(\eta\sigma)$. Similarly, the functorial properties of $r_P(\tilde{w}^{-1}, \cdot)$ implies that  $\Ad(r_P(\tilde{w}^{-1}, \tilde{w}\sigma))$ induces the identity map $S^M(\tilde{w}\sigma) \rightiso S^M(\sigma)$. One readily checks that $\Ad(I^G_P(i))$ induces $\Ad(i): S^M(\eta\sigma) \rightiso S^M(\tilde{w}\sigma)$, hence the assertion follows.
\end{proof}

Before stating the next result, recall that $\sigma$, $\tilde{w}\sigma$ and $\eta\sigma$ share the same underlying space $V_\sigma$.

\begin{lemma}\label{prop:rho-sigma-correspondence}
  Let $\rho \in \Pi_-(S^M(\sigma))$ and $\sigma^\sharp \in \Pi_\sigma$. By identifying the groups $S^M(\sigma)$, $S^M(\tilde{w}\sigma)$ and $S^M(\eta\sigma)$, the following are equivalent:
  \begin{enumerate}
    \item $\rho \in \Pi_-(S^M(\sigma))$ corresponds to $\sigma^\sharp \hookrightarrow \sigma|_{M^\sharp}$;
    \item $\rho \in \Pi_-(S^M(\tilde{w}\sigma))$ corresponds to $\tilde{w}\sigma^\sharp \hookrightarrow \tilde{w}\sigma|_{M^\sharp}$;
    \item $\tilde{w}\rho \in \Pi_-(S^M(\eta\sigma))$ corresponds to $\tilde{w}\sigma^\sharp \hookrightarrow \eta\sigma|_{M^\sharp}$;
    \item $\tilde{w}\rho \in \Pi_-(S^M(\sigma))$ corresponds to $\tilde{w}\sigma^\sharp \hookrightarrow \sigma|_{M^\sharp}$.
  \end{enumerate}
\end{lemma}
\begin{proof}
  Recall that $\rho$ corresponds to $\sigma^\sharp \hookrightarrow \sigma|_{M^\sharp}$ means that $\rho \boxtimes \sigma^\sharp \hookrightarrow \mathfrak{S}(\sigma)$, where $\mathfrak{S}(\sigma)$ is the $S^M(\sigma) \times M^\sharp(F)$-representation on $V_\sigma$ defined in Theorem \ref{prop:S-decomp}.

  The first two properties are equivalent by a transport of structure via $\Ad(\tilde{w}): M^\sharp \to M^\sharp$. The last two properties are evidently equivalent. Finally, the equivalence between the second and the third properties follows by pulling $\rho$ back via $\Ad(i): S^M(\eta\sigma) \rightiso S^M(\tilde{w}\sigma)$.
\end{proof}

Now, recall that $Z^M(\sigma)$ is the projection to $X^M(\sigma)$ of the center of $S^M(\sigma)$. Temporarily fix a preimage $I^G_\eta$ for every $\eta \in X^G(\pi)$ and define
\begin{gather}\label{eqn:perp}
  Z^M(\sigma)^\perp := \left\{ \eta \in X^G(\pi) : \forall \omega \in Z^M(\sigma), \; I^G_\eta I^G_\omega = I^G_\omega I^G_\eta \right\} \supset X^M(\sigma).
\end{gather}

\begin{proposition}\label{prop:Z-perp}
  Let $\sigma$, $w$, $\eta$ as before. For any $\sigma^\sharp \in \Pi_\sigma$ (resp. $\rho \in \Pi_-(S^M(\sigma))$), we have $\tilde{w}\sigma^\sharp \simeq \sigma^\sharp$ (resp. $\tilde{w}\rho \simeq \rho$) if and only if $\eta \in Z^M(\sigma)^\perp$.
\end{proposition}
\begin{proof}
  Let $\rho \in \Pi_-(S^M(\sigma))$ be the representation corresponding to $\sigma^\sharp \in \Pi_\sigma$ by Theorem \ref{prop:S-decomp}. By Lemma \ref{prop:rho-sigma-correspondence}, it suffices to show that $\tilde{w}\rho \simeq \rho$ if and only if $\eta \in Z_M(\sigma)^\perp$.

  The elements in $\Pi_-(S^M(\sigma))$ are described by a variant of the Stone-von Neumann theorem for the central extension $1 \to \C^\times \to S^M(\sigma) \to X^M(\sigma) \to 1$. Namely, consider the data $(L, \rho_0)$ where
  \begin{itemize}
    \item $L$: a maximal abelian subgroup of $S^M(\sigma)$;
    \item $\rho_0$: an irreducible representation of $L$ such that $\rho_0(z)=z$ for all $z \in \C^\times \subset L$.
  \end{itemize}
  Then $\rho := \text{Ind}^{S^M(\sigma)}_L(\rho_0)$ is an element of $\Pi_-(S^M(\sigma))$. Every $\rho \in \Pi_-(S^M(\sigma))$ arises in this way. Moreover, the isomorphism class of $\rho$ is determined by its central character. These facts are standard consequences of Mackey's theory. See \cite[0.3]{KP84} and the remark after Proposition \ref{prop:K_0-diagram}.

  We have $\tilde{w}\rho = \rho \circ \Ad(I^G_\eta)$ by Lemma \ref{prop:w-eta}. To conclude the proof, it suffices to show that $\Ad(I^G_\eta)$ fixes the central character of $\rho$ if and only if $\eta \in Z^M(\sigma)^\perp$. This is immediate.
\end{proof}

\begin{corollary}\label{prop:L-sharp}
  Assume $\sigma \in \Pi_{2,\mathrm{temp}}(M)$ and $\sigma^\sharp \in \Pi_\sigma$. Then we have $L(\sigma^\sharp)=Z^M(\sigma)^\perp$, and the map $\Gamma$ in Proposition \ref{prop:Goldberg} is an isomorphism
  \begin{align*}
    \Gamma: R_{\sigma^\sharp} & \longrightarrow Z^M(\sigma)^\perp/X^M(\sigma) \\
    w W^0_{\sigma^\sharp} & \longmapsto \eta X^M(\sigma)
  \end{align*}
  where $w \in W_{\sigma^\sharp}$ and $\eta \in Z^M(\sigma)^\perp$ satisfy the relation
  $$ w\sigma \simeq \eta\sigma. $$
\end{corollary}
\begin{proof}
  This results immediately from the definition of $L(\sigma^\sharp)$.
\end{proof}

\subsection{Cocycles}\label{sec:cocycle}
\begin{definition}\label{def:obstruction}
  Suppose for a moment that $H$ is a finite group and $N$ is a normal subgroup of $H$. Let $\rho$ be an irreducible representation of $N$ and assume that $h\rho := \rho \circ \Ad(h)^{-1} \simeq \rho$ for all $h \in H$. This is a necessary condition for extending $\rho$ to an irreducible representation of $H$, but not sufficient in general. Recall the following construction of an obstruction $\mathbf{c}_\rho \in H^2(H/N, \C^\times)$ for extending $\rho$, where $\C^\times$ is equipped with the trivial $H/N$-action. We can choose intertwining operators $\rho(h) \in \Isom_N(h\rho, \rho)$ for each $h \in H$, such that
  \begin{align*}
    \rho(nh) & = \rho(n) \rho(h), \\
    \rho(hn) & = \rho(h) \rho(n)
  \end{align*}
  for every $h \in H$, $n \in N$. Note that either of the equations above implies the other. There is a $\C^\times$-valued $2$-cocycle $c_\rho$ characterized by
  \begin{gather}\label{eqn:obstruction}
    \rho(h_1 h_2) = c_\rho(h_1, h_2) \rho(h_1) \rho(h_2), \quad h_1, h_2 \in H.
  \end{gather}

  One readily checks that $c_\rho$ factors through $H/N \times H/N$, thus defines a class $\mathbf{c}_\rho \in H^2(H/N, \C^\times)$. This cohomology class only depends on $\rho$ itself.

  The formalism can also be generalized to the case where $H,N$ are central extensions of finite groups by $\C^\times$, and $\rho(z) = z\cdot\identity$ for all $z \in \C^\times$.
\end{definition}

Let us return to the formalism of the previous subsection. In particular, we assume $P=MU \subset G$ and $\sigma \in \Pi_{2,\text{temp}}(M)$ with the underlying vector space $V_\sigma$. Set $\pi := I^G_P(\sigma)$ as usual. For every $\eta \in Z^M(\sigma)^\perp$, we fix $w \in W(M)$, a representative $\tilde{w} \in G^\sharp(F)$, and an isomorphism
$$ i: \eta\sigma \rightiso \tilde{w}\sigma. $$

Let $\rho \in \Pi_-(S^M(\sigma))$ be corresponding to $\sigma^\sharp \in \Pi_\sigma$. Proposition \ref{prop:Z-perp} implies that $\tilde{w}\rho \simeq \rho$ for every $\eta$ as above. Equivalently, $\rho \circ \Ad((I^G_\eta)^{-1}) \simeq \rho$ for every $I^G_\eta \in S^G(\pi)$ in the preimage of $\eta$ by Lemma \ref{prop:w-eta}. We will use the shorthand
$$ {}^\eta \rho := \rho \circ \Ad((I^G_\eta)^{-1}). $$

As in Definition \ref{def:obstruction}, one considers the problem of extending $\rho$ to the preimage of $Z^M(\sigma)^\perp$ in $S^G(\pi)$. Recall that $Z^M(\sigma)^\perp/X^M(\sigma) = R_{\sigma^\sharp}$ by Corollary \ref{prop:L-sharp}. The goal of this subsection is to describe the obstruction class $\mathbf{c}_\rho \in H^2(R_{\sigma^\sharp}, \C^\times)$ so-obtained.

Recall that in Theorem \ref{prop:S-decomp}, we have defined an $S^M(\sigma) \times M^\sharp(F)$-representation $\mathfrak{S} = \mathfrak{S}(\sigma)$ on $V_\sigma$. Analogously, we define $\mathfrak{S}(\eta\sigma)$ and $\mathfrak{S}(\tilde{w}\sigma)$; all of them are realized on $V_\sigma$. We fix an embedding $\iota: \rho \boxtimes \sigma^\sharp \hookrightarrow \mathfrak{S}(\sigma)$ of $S^M(\sigma) \times M^\sharp(F)$-representations. By Lemma \ref{prop:rho-sigma-correspondence}, the same map gives $\iota: \rho \boxtimes \tilde{w}\sigma^\sharp \hookrightarrow \mathfrak{S}(\tilde{w}\sigma)$ and $\iota: \rho \boxtimes \sigma^\sharp \hookrightarrow \mathfrak{S}(\eta\sigma)$ with appropriate equivariances.

\begin{lemma}\label{prop:existential}
  For $\eta$, $\tilde{w}$, $\sigma^\sharp$ fixed as before, we define $\mathfrak{S}'(\eta\sigma)$ to be the $S^M(\eta\sigma) \times M^\sharp(F)$-representation on $V_\sigma$ defined by
  $$ \mathfrak{S}'(\eta\sigma)(I,x) = \mathfrak{S}(\eta\sigma)(\Ad(I^G_\eta)^{-1} I, x), \quad I \in S^M(\eta\sigma), x \in M^\sharp(F). $$
  Then the map $\iota$ induces an embedding of $S^M(\eta\sigma) \times M^\sharp(F)$-representations
  $$ \iota: {}^\eta \rho \boxtimes \sigma^\sharp \hookrightarrow \mathfrak{S}'(\eta\sigma) $$
  and there exists a unique equivariant isomorphism
  $$ \alpha \boxtimes \sigma^\sharp(\tilde{w})^{-1}: {}^\eta \rho \boxtimes \sigma^\sharp \rightiso \rho \boxtimes \tilde{w}\sigma^\sharp, $$
  for some $\alpha \in \Isom_{S^M(\sigma)}({}^\eta \rho, \rho)$ and $\sigma^\sharp(\tilde{w}) \in \Isom_{M^\sharp}(\tilde{w}\sigma^\sharp, \sigma^\sharp)$, which makes the following diagram commutative.
  $$\xymatrix{
    \mathfrak{S}'(\eta\sigma) \ar[rr]^{i}_{\simeq} & & \mathfrak{S}(\tilde{w}\sigma) \\
    {}^\eta \rho \boxtimes \sigma^\sharp \ar[u]^{\iota} \ar[rr]_{\alpha \boxtimes \sigma^\sharp(\tilde{w})^{-1}} & & \rho \boxtimes \tilde{w}\sigma^\sharp \ar[u]_{\iota}
  }$$
\end{lemma}
Observe that the pair $(\alpha, \sigma^\sharp(\tilde{w})^{-1})$ is unique up to $\{(z,z^{-1}) : z \in \C^\times \}$.

\begin{proof}
  The $S^M(\eta\sigma)$-action on $\mathfrak{S}'(\eta\sigma)$ makes $i$ equivariant. The leftmost vertical arrow comes from the original embedding $\iota: \rho \boxtimes \sigma^\sharp \hookrightarrow \mathfrak{S}(\eta\sigma)$ by an $\Ad(I^G_\eta)^{-1}$-twist. The images of the vertical arrows are characterized as the $\sigma^\sharp$ (resp. $\tilde{w}\sigma^\sharp$) -isotypic parts under the $M^\sharp(F)$-action. Proposition \ref{prop:Z-perp} implies that $\sigma^\sharp \simeq \tilde{w}\sigma^\sharp$. Therefore there must exist an equivariant isomorphism ${}^\eta \rho \boxtimes \sigma^\sharp \rightiso \rho \boxtimes \tilde{w}\sigma^\sharp$ that makes the diagram commute. Such an isomorphism must be of the form $\alpha \boxtimes \sigma^\sharp(\tilde{w})^{-1}$.
\end{proof}

\begin{lemma}\label{prop:intertwining}
  Write $r_P := r_P(\tilde{w}, \sigma)$ and $r_{P^\sharp} := r_{P^\sharp}(\tilde{w}, \sigma^\sharp)$. There is a commutative diagram
  $$\xymatrix{
    I^G_P(\sigma) \ar[r]^{r_P} & I^G_P(\tilde{w}\sigma) \\
    \rho \boxtimes I^{G^\sharp}_{P^\sharp}(\sigma^\sharp) \ar[r]_{\identity \boxtimes r_{P^\sharp}} \ar[u]^{I^{G^\sharp}_{P^\sharp}(\iota)} & \rho \boxtimes I^{G^\sharp}_{P^\sharp}(\tilde{w}\sigma^\sharp) \ar[u]_{I^{G^\sharp}_{P^\sharp}(\iota)}
  }$$
  whose arrows are equivariant for the $S^M(\tilde{w}\sigma) \times G^\sharp(F)$ and $S^M(\sigma) \times G^\sharp(F)$-actions.
\end{lemma}
\begin{proof}
  Without loss of generality we may assume $\rho = \Hom_{M^\sharp}(\sigma^\sharp, \sigma)$, i.e.\! the multiplicity space. The embedding $\iota: \rho \boxtimes \sigma^\sharp \hookrightarrow \sigma$ can be taken to be $\epsilon \otimes v \mapsto \epsilon(v)$. Then the commutativity of the diagram follows by applying Theorem \ref{prop:normalizing-invariance} to each $\epsilon \in \Hom_{M^\sharp}(\sigma^\sharp, \sigma)$. The equivariance of the horizontal arrows results from Theorem \ref{prop:normalizing-invariance} and the functorial properties of $r_P(\tilde{w},\cdot)$, $r_{P^\sharp}(\tilde{w}, \cdot)$.
\end{proof}

\begin{lemma}\label{prop:pull-back}
  With the notations of Lemma \ref{prop:existential}, there is a commutative diagram
  $$\xymatrix{
    I^G_P(\eta\sigma) \ar[rr]^{r_P(\tilde{w}, \sigma)^{-1} \circ I^G_P(i)} & & I^G_P(\sigma) \\
    {}^\eta\rho \boxtimes I^{G^\sharp}_{P^\sharp}(\sigma^\sharp) \ar[rr]_{\alpha \boxtimes R_{P^\sharp}(\tilde{w}, \sigma^\sharp)^{-1}} \ar[u]^{I^{G^\sharp}_{P^\sharp}(\iota)} & & \rho \boxtimes I^{G^\sharp}_{P^\sharp}(\sigma^\sharp) \ar[u]_{I^{G^\sharp}_{P^\sharp}(\iota)}
  }$$
  where we set $R_{P^\sharp}(\tilde{w}, \sigma^\sharp) := \sigma^\sharp(\tilde{w}) \circ r_{P^\sharp}(\tilde{w}, \sigma^\sharp)$, by using the pair $(\alpha, \sigma^\sharp(\tilde{w})^{-1})$ of isomorphisms in Lemma \ref{prop:existential}.
\end{lemma}
\begin{proof}
  This is the concatenation of the diagram in Lemma \ref{prop:intertwining} and that of Lemma \ref{prop:existential}, after applying $I^{G^\sharp}_{P^\sharp}(\cdot)$.
\end{proof}

\begin{proposition}\label{prop:cocycle}
  Let $\mathbf{c}_\rho$ be the obstruction of extending $\rho$ to the preimage of $Z^M(\sigma)^\perp$ in $S^G(\pi)$, and $\mathbf{c}_{\sigma^\sharp}$ be the class attached to $R_{\sigma^\sharp}$ in \eqref{eqn:R-cocycle}. Then we have
  $$ \mathbf{c}_\rho = \mathbf{c}_{\sigma^\sharp}^{-1} $$
  in $H^2(R_{\sigma^\sharp}, \C^\times)$.
\end{proposition}
\begin{proof}
  Fix $\iota: \rho \boxtimes \sigma^\sharp \hookrightarrow \sigma$. Also fix a set of representatives $\tilde{w} \in G^\sharp(F)$ for each $w \in R_{\sigma^\sharp}$. For each $\eta$, together with the auxiliary choice $i: \eta\sigma \rightiso \tilde{w}\sigma$, the top row of the diagram in Lemma \ref{prop:pull-back} gives an operator $I^G_\eta \circ A_\eta^{-1}: I^G_P(\eta\sigma) \rightiso I^G_P(\sigma)$ for some $I^G_\eta \in S^G(\pi)$. The isomorphism $A_\eta: \eta I^G_P(\sigma) \rightiso I^G_P(\sigma)$ has no effect after restriction. Therefore Lemma \ref{prop:pull-back} asserts that $I^G_\eta$ is pulled-back to $\alpha \boxtimes R_{P^\sharp}(\tilde{w}, \sigma^\sharp)^{-1}$ under $I^{G^\sharp}_{P^\sharp}(\iota)$.

  Now we can forget $i$ and vary $I^G_\eta$ in the preimage of $\eta$ in $S^G(\pi)$, which is a $\C^\times$-torsor. Regard $\alpha = \alpha(I^G_\eta)$ as a function of $I^G_\eta$; it is well-defined once we have pinned down the operator $\sigma^\sharp(\tilde{w})$ coupled with $\alpha$.

  Suppose that $\eta$ is replaced by $\eta\omega$ where $\omega \in X^M(\sigma)$; accordingly, $I^G_\eta$ is replaced by $I^G_\eta I^G_\omega$ where $I^M_\omega \in S^M(\sigma)$ lies in the preimage of $\omega$ and $I^M_\omega \mapsto I^G_\omega$. This does not affect the chosen data $\tilde{w}$ and $\iota$. On the other hand, the diagram in Lemma \ref{prop:pull-back} says that $\alpha \boxtimes R_{P^\sharp}(\tilde{w}, \sigma^\sharp)^{-1}$ is replaced by
  $$ \alpha \circ \rho(I^M_\omega) \boxtimes R_{P^\sharp}(\tilde{w}, \sigma^\sharp)^{-1}. $$

  It follows that we can pin down the operators $\sigma^\sharp(\tilde{w})$, and well-define a function
  $$ I^G_\eta \mapsto \alpha(I^G_\eta) \in \Isom_{S^M(\sigma)}({}^\eta \rho, \rho), $$
  for every $I^G_\eta$ in the preimage of $\eta \in Z^M(\sigma)^\perp$ in $S^G(\pi)$, such that
  \begin{itemize}
    \item $I^G_\eta$ is pulled-back to $\alpha(I^G_\eta) \boxtimes R_{P^\sharp}(\tilde{w}, \sigma^\sharp)^{-1}$ under $I^{G^\sharp}_{P^\sharp}(\iota)$;
    \item $\alpha(I^G_\eta I^G_\omega) = \alpha(I^G_\eta) \rho(I^M_\omega)$ for every $\omega \in X^M(\sigma)$ and $I^M_\omega$ in its preimage.
  \end{itemize}

  Such a family of intertwining operators meets the requirements of Definition \ref{def:obstruction}, thus the obstruction can be accounted by the $\C^\times$-valued $2$-cocycle $c_\rho$ given by
  $$ \alpha(I^G_\xi I^G_\eta) = c_\rho(w_\xi, w_\eta) \alpha(I^G_\xi) \alpha(I^G_\eta), \quad \xi, \eta \in Z^M(\sigma)^\perp, $$
  where $w_\eta \in R_{\sigma^\sharp}$ denotes the element determined by $\eta$ as in Corollary \ref{prop:L-sharp}. Idem for $w_\xi$.

  On the other hand, write $w_\eta \mapsto \tilde{w}_\eta$ for the map that picks the chosen representative for $w_\eta \in R_{\sigma^\sharp}$. The equation \eqref{eqn:R-cocycle} defines $2$-cocycle $c_{\sigma^\sharp}$. For every $\xi, \eta \in Z^M(\sigma)^\perp$ we obtain
  $$ R_{P^\sharp}(\tilde{w}_{\xi\eta}, \sigma^\sharp) = R_{P^\sharp}(\tilde{w}_{\eta\xi}, \sigma^\sharp) = c_{\sigma^\sharp}(w_\eta, w_\xi) R_{P^\sharp}(\tilde{w}_\eta, \sigma^\sharp) R_{P^\sharp}(\tilde{w}_\xi, \sigma^\sharp). $$

  All in all, the pull-back of $I^G_\xi I^G_\eta$ by $I^{G^\sharp}_{P^\sharp}(\iota)$ equals
  $$ c_\rho(w_\xi, w_\eta) c_{\sigma^\sharp}(w_\eta, w_\xi)^{-1} \cdot (\text{the pull-back of } I^G_\xi) \circ (\text{the-pull back of } I^G_\eta ). $$

  Therefore $c_\rho(w_\xi, w_\eta) = c'_{\sigma^\sharp}(w_\xi, w_\eta) := c_{\sigma^\sharp}(w_\eta, w_\xi)$. It is routine to check that $c'_{\sigma^\sharp}: (R_{\sigma^\sharp})^2 \to \C^\times$ is also a $2$-cocycle. Denote by $\mathbf{c}'_{\sigma^\sharp}$ the cohomology class of $c'_{\sigma^\sharp}$. It remains to show that $\mathbf{c}'_{\sigma^\sharp} = \mathbf{c}^{-1}_{\sigma^\sharp}$. We use the following observation: let $A$ be finite abelian group acting trivially on $\C^\times$, we claim that there is an injective group homomorphism
  \begin{align*}
    \text{comm}: H^2(A, \C^\times) & \longrightarrow \Hom\left( \bigwedge^2 A, \C^\times \right), \\
    \mathbf{c} & \longmapsto [x \wedge y \mapsto c(y,x)c(x,y)^{-1}]
  \end{align*}
  where $c$ is any $2$-cocycle representing the class $\mathbf{c}$. Indeed, let
  $$ 1 \to \C^\times \to \tilde{A} \to A \to 1 $$
  be the central extension corresponding to $\mathbf{c}$, then $(x,y) \mapsto c(y,x)c(x,y)^{-1}$ is just the commutator pairing of this central extension. The injectivity results from the elementary fact that such an extension splits if and only if $\tilde{A}$ is commutative.

  Apply this to $A = R_{\sigma^\sharp}$. Since $\text{comm}(\mathbf{c}'_{\sigma^\sharp}) = \text{comm}(\mathbf{c}^{-1}_{\sigma^\sharp})$, we deduce $\mathbf{c}'_{\sigma^\sharp} = \mathbf{c}^{-1}_{\sigma^\sharp}$, as asserted.
\end{proof}

\section{The inner forms of $\SL(N)$}\label{sec:SL}
\subsection{The groups}\label{sec:groups}
Fix $N \in \Z_{\geq 1}$ and let $G^* := \GL_F(N)$. Let $A$ be a central simple algebra over $F$ of dimension $N^2$. There exist $n \in \Z_{\geq 1}$ and a central division algebra $D$ over $F$ satisfying
$$ n^2 \cdot \dim_F D = N^2, $$
such that $A$ is isomorphic to $\End_D(D^n)$. The division $F$-algebra $D$ is uniquely determined by $A$. We put
\begin{align*}
  \Nrd & := \text{ the reduced norm of } A, \\
  \GL_D(n) & := A^\times, \\
  \SL_D(n) & := \Ker(\Nrd: A^\times \to \Gm).
\end{align*}

We can regard $A^\times$ as a reductive $F$-group. It is well-known that $A \mapsto A^\times$ induces a bijection between the central simple $F$-algebras of dimension $N^2$ and the inner forms of $G^*$. Given $A$, or equivalently given $(n,D)$ as above, we shall always write
$$ G := \GL_D(n). $$
Under an inner twist $\psi: G \times_F \bar{F} \rightiso G^* \times_G \bar{F}$, the determinant map $\det: G^* \to \Gm$ corresponds to $\Nrd: G \to \Gm$. Since the parametrization of the inner forms of an $F$-group $G^*$ only depends on $G^*_\text{AD}$, the map $A \mapsto \SL_D(n)$ establishes a bijection between the central simple $F$-algebras of dimension $N^2$ and the inner forms of $\SL_N(F)$. We write
$$ G^\sharp := \SL_D(n) = G_\text{der}. $$

Note that $G(F)/G^\sharp(F) = (G/G^\sharp)(F) = F^\times$, since $H^1(F, G^\sharp)$ is trivial by Hasse principle.

As mentioned in \S\ref{sec:res-L}, the inner twist $\psi$ gives a correspondence between Levi subgroups: the Levi subgroups of $G$ is of the form
$$ M = \prod_{i=1}^r \GL_D(n_i), \quad n_1 + \cdots + n_r = n. $$
and the corresponding Levi subgroup of $G^*$, well-defined up to conjugacy, is simply
$$ M^* = \prod_{i=1}^r \GL_F(n_i \cdot \dim_F D). $$

The L-groups of $G$ and $G^\sharp$ are easily described. We have
\begin{align*}
  \hat{G} = \hat{G^*} & = \GL(N,\C), \\
  \widehat{G^\sharp} & = \PGL(N,\C), \\
  \hat{G}_\text{SC} = (\widehat{G^\sharp})_\text{SC} & = \SL(N,\C), \\
  Z_{\hat{G}_\text{SC}} = Z_{(\widehat{G^\sharp})_\text{SC}} & = \mu_N(\C) \\
  & := \{z \in \C^\times : z^N = 1 \}.
\end{align*}
These complex groups are endowed with the trivial Galois action, thus $\Lgrp{G} = \hat{G} \times W_F$ and $\Lgrp{G^\sharp} = \widehat{G^\sharp} \times W_F$. The inclusion $G^\sharp \hookrightarrow G$ is dual to the quotient homomorphism $\GL(N,\C) \to \PGL(N,\C)$.

It is also possible to describe the characters $\chi_G = \chi_{G^\sharp}$ in \eqref{eqn:chi-G} explicitly. Observe that $\Gamma_F$ acts trivially on $Z_{\hat{G}_\text{SC}}$, and one can identify the Pontryagin dual of $Z_{\hat{G}_\text{SC}} = \mu_N(\C)$, denoted by $Z_{\hat{G}_\text{SC}}^D$, with $\Z/N\Z$: a class $e \in \Z/N\Z$ corresponds to the character $z \mapsto z^e$. For the inner form $G = \GL_D(n)$ of $G^* = \GL_F(N)$, we have
\begin{gather}\label{eqn:chi-G-special}
  \chi_G \in Z_{\hat{G}_\text{SC}}^D \; \text{ corresponds to } (n \text{ mod } N) \in \Z/N\Z .
\end{gather}

Later on, the results of \S\ref{sec:res-2} will be applied to the tempered representations of $G(F)$. This is justified by the following general result.

\begin{theorem}[Sécherre {\cite{Sec09}}]\label{prop:irred}
  Let $P=MU$ be a parabolic subgroup of $G$ and $\sigma \in \Pi_{\mathrm{unit}}(M)$, then $I^G_P(\sigma)$ is irreducible. In particular, Hypothesis \ref{hyp:irred} is satisfied by $\sigma$.
\end{theorem}

Note that the tempered case is already established in \cite{DKV84}.

\subsection{Local Langlands correspondences}\label{sec:LLC}
This subsection is a summary of \cite[Chapter 11]{HS12}.

\paragraph{Local Langlands correspondence for $\GL_D(n)$}
Using the local Langlands correspondence for $G^*$, we can define the notion of $G^*$-generic elements in $\Phi(G)$: a parameter $\phi \in \Phi(G) \subset \Phi(G^*)$ is called $G^*$-generic if it parametrizes a generic representation of $G^*(F)$. This defines a subset $\Phi_{G^*-\text{gen}}(G)$ of $\Phi(G)$.

\begin{theorem}[{\cite[Lemma 11.1 and 11.2]{HS12}}]\label{prop:LLC-G}
  Let $G = \GL_D(n)$ and $G^* = \GL_F(N)$ as in \S\ref{sec:groups}. There exist a subset $\Pi_{G^*-\mathrm{gen}}(G)$ of $\Pi(G)$ satisfying
  \begin{itemize}
    \item $\Pi_{G^*-\mathrm{gen}}(G) \supset \Pi_{\mathrm{temp}}(G)$,
    \item $\Pi_{G^*-\mathrm{gen}}(G)$ is stable under twists by $(G(F)/G^\sharp(F))^D$,
  \end{itemize}
  and a canonically defined bijection between $\Pi_{G^*-\mathrm{gen}}(G)$ and $\Phi_{G^*-\mathrm{gen}}(G)$, denoted by $\pi \leftrightarrow \phi$, such that
  $$\xymatrix{
    \Pi_{G^*-\mathrm{gen}}(G) \ar@{<->}[r]^{\sim} & \Phi_{G^*-\mathrm{gen}}(G) \\
    \Pi_{\mathrm{temp}}(G) \ar@{<->}[r]^{\sim} \ar@{^{(}->}[u] & \Phi_{\mathrm{bdd}}(G) \ar@{^{(}->}[u] \\
    \Pi_{2,\mathrm{temp}}(G) \ar@{<->}[r]^{\sim} \ar@{^{(}->}[u] & \Phi_{2,\mathrm{bdd}}(G) \ar@{^{(}->}[u]
  }$$

  The correspondence satisfies the following compatibility properties.
  \begin{enumerate}
    \item When $G=G^*$, the usual Langlands correspondence for $\GL_F(N)$ is recovered.
    \item Given $\pi \leftrightarrow \phi$ and $\mathbf{a} \in H^1_{\mathrm{cont}}(W_F, Z_{\hat{G}})$, let $\eta$ be the character of $G(F)$ deduced from $\mathbf{a}$ by local class field theory, then we have $\omega\pi \leftrightarrow a\phi$.
    \item Given a Levi subgroup $M = \prod_{i \in I} \GL_D(n_i)$ of $G$ and let $\sigma := \boxtimes_{i \in I} \sigma_i \in \Pi_{\mathrm{temp}}(M)$. Let $\phi_M \in \Phi_{\mathrm{bdd}}(M)$ such that $\sigma \leftrightarrow \phi_M$ and let $\phi$ be the composition of $\phi_M$ with some L-embedding $\Lgrp{M} \hookrightarrow \Lgrp{G}$. Then for any $P \in \mathcal{P}(M)$ we have
    $$ I^G_P(\sigma) \leftrightarrow \phi .$$
  \end{enumerate}
\end{theorem}
Note that in the last assertion, $I^G_P(\sigma)$ is irreducible according to Theorem \ref{prop:irred}.

The definitions of $\Pi_{G^*-\text{gen}}(G)$ and $\pi \leftrightarrow \phi$ are based upon the local Langlands correspondence for $G^*$ and the Jacquet-Langlands correspondence for essentially square-integrable representations. We refer the reader to \cite[\S 11]{HS12} for details; the compatibility properties are also implicit therein. Only the tempered/bounded case of the theorem will be used in this article.

\paragraph{Local Langlands correspondence for $\SL_D(n)$}
Let $G = \GL_D(n)$ and $G^\sharp = G_\text{der} = \SL_D(n)$, so that the formalism in \S\ref{sec:res} is applicable. The idea is to define the packets $\Pi_{\phi^\sharp}$ via restriction, by combining the results in \S\ref{sec:res-rep} and \S\ref{sec:res-L}. Let $\Phi_{G^*-\text{gen}}(G^\sharp)$ be the set of $\phi^\sharp \in \Phi(G^\sharp)$ such that $\phi \in \Phi_{G^*-\text{gen}}(G)$ for some lifting $\phi$ of $\phi^\sharp$ (hence for all liftings, since twisting by characters does not affect $G^*$-genericity). For any $\phi^\sharp \in \Phi_{G^*-\text{gen}}(G^\sharp)$, define the corresponding packet by
$$ \Pi_{\phi^\sharp} := \Pi_\pi, \quad \pi \leftrightarrow \phi \text{ for some lifting } \phi \in \Phi_{G^*-\text{gen}}(G). $$
By Proposition \ref{prop:res-disjoint} and Theorem \ref{prop:lifting-parameter}, the definition of $\Pi_{\phi^\sharp}$ does not depend on the choice of lifting.

On the other hand, set
$$ \Pi_{G^*-\text{gen}}(G^\sharp) = \bigsqcup_{\pi} \Pi_\pi $$
where $\pi$ ranges over the $(G(F)/G^\sharp(F))^D$-orbits in $\Pi_{G^*-\text{gen}}(G)$. Our version of the local Langlands correspondence for $G^\sharp$ is stated as follows.

\begin{theorem}[{\cite[Chapter 12]{HS12}}]\label{prop:LLC-SL}
  We have $\Pi_\mathrm{temp}(G^\sharp) \subset \Pi_{G^*-\mathrm{gen}}(G^\sharp)$, and there is a decomposition
  \begin{gather}\label{eqn:LLC-SL}
    \Pi_{G^*-\mathrm{gen}}(G^\sharp) = \bigsqcup_{\phi^\sharp \in \Phi_{G^*-\mathrm{gen}}(G^\sharp)} \Pi_{\phi^\sharp}
  \end{gather}
  which restricts to
  \begin{align*}
    \Pi_{\mathrm{temp}}(G^\sharp) = \bigsqcup_{\phi^\sharp \in \Phi_{\mathrm{temp}}(G^\sharp)} \Pi_{\phi^\sharp}, \\
    \Pi_{2,\mathrm{temp}}(G^\sharp) = \bigsqcup_{\phi^\sharp \in \Phi_{2,\mathrm{temp}}(G^\sharp)} \Pi_{\phi^\sharp}.
  \end{align*}
\end{theorem}
\begin{proof}
  This follows from Proposition \ref{prop:heredity}, Theorem \ref{prop:LLC-G}, \ref{prop:lifting-parameter} and Proposition \ref{prop:res-disjoint}.
\end{proof}

Note that each packet $\Pi_{\phi^\sharp}$ is finite. From the endoscopic point of view, in order to justify the correspondence \eqref{eqn:LLC-SL}, one has to explicate
\begin{enumerate}
  \item the internal structure of the packets $\Pi_{\phi^\sharp}$;
  \item their relation to $S$-groups;
  \item the endoscopic character identities for $G^\sharp$.
\end{enumerate}

We will recall the definition of $S$-groups (or more precisely their component groups, called $\mathscr{S}$-groups...) in the next subsection, then summarize its relation to the internal structure of packets; this is one of the main results in \cite{HS12}. The character identities will not be used in this article; we refer the interested reader to \cite[Theorem 12.7]{HS12}.

\paragraph{Normalizing factors}
Choose a non-trivial additive character $\psi_F: F \to \C^\times$. Now we can exhibit a canonical family of normalizing factors for $G$ and $G^\sharp$ with respect to $\psi_F$.

Let us begin with $G$. According to the construction in Remark \ref{rem:construction-normalization}, it suffices to consider the case of inducing representations $\sigma \in \Pi_{2,\text{temp}}(M)$ where $M$ is a Levi subgroup of $G$. When $D=F$ or equivalently $G=G^*=\GL_F(N)$, the formula in Remark \ref{rem:normalization} furnishes a family of normalizing factors in the tempered case, by the Langlands-Shahidi method. To pass to the non-quasisplit case, we use the preservation of $\mu$-functions by Jacquet-Langlands correspondence \cite[Theorem 7.2]{AP05} (up to a harmless constant depending only on $D$ and $n$).

From Theorem \ref{prop:normalizing-invariance}, we deduce a canonical family of normalizing factors for $G^\sharp$, at least for the inducing representations $\sigma^\sharp$ whose central character is unitary. In what follows, the normalized intertwining operators for $G$ and $G^\sharp$ are assumed to be defined with respect to these factors.

\subsection{Identification of $S$-groups}\label{sec:iden-S-groups}
\paragraph{Generalities}
To begin with, we summarize the definition of the $S$-groups in the non-quasisplit case by following \cite{Ar06}.

\begin{definition}
  Let $G$ be a connected reductive $F$-group. Choose a quasisplit inner twist $\psi: G \times_F \bar{F} \to G^* \times_F \bar{F}$ as well as an $F$-splitting for $G^*(\bar{F})$ to define the L-groups. Let $\phi \in \Phi(G^*)$, we set
  \begin{align*}
    S_{\phi, \text{ad}} & := Z_{\hat{G}}(\Im(\phi))/Z^{\Gamma_F}_{\hat{G}} \\
    & \rightiso \left( Z_{\hat{G}}(\Im(\phi)) Z_{\hat{G}} \right) / Z_{\hat{G}} \; \subset \hat{G}_\text{AD}, \\
    S_{\phi, \text{sc}} & := \text{ the preimage of } S_{\phi, \text{ad}} \text{ in } \hat{G}_\text{SC}, \\
    \mathscr{S}_{\phi} & := \pi_0(S_{\phi, \text{ad}}, 1), \\
    \tilde{\mathscr{S}}_{\phi} & := \pi_0(S_{\phi, \text{sc}}, 1).
  \end{align*}
  From the central extension $1 \to Z_{\hat{G}_\text{SC}} \to S_{\phi, \text{sc}} \to S_{\phi, \text{ad}} \to 1$, we obtain another central extension
  $$ 1 \to \tilde{Z}_\phi \to \tilde{\mathscr{S}}_{\phi} \to \mathscr{S}_\phi \to 1 $$
  where
  $$ \tilde{Z}_\phi := Z_{\hat{G}_\text{SC}}/(Z_{\hat{G}_\text{SC}} \cap S^0_{\phi, \text{sc}}) = \Im[  Z_{\hat{G}_\text{SC}} \to \tilde{\mathscr{S}}_{\phi} ]. $$
\end{definition}

\begin{remark}
  When $G$ is an inner form of $\SL(N)$, we recover the definition of the modified $S$-groups in \cite{HS12}.
\end{remark}

The relevance condition of L-parameters intervenes in the following result. Recall that we have defined a character $\chi_G$ of $Z^{\Gamma_F}_{\hat{G}_\text{SC}}$ in \eqref{eqn:chi-G}.

\begin{lemma}[{\cite[Lemma 9.1]{HS12}}]\label{prop:chi-G-factorization}
  If $\phi \in \Phi(G)$, then $\chi_G : Z^{\Gamma_F}_{\hat{G}_\text{SC}} \to \C^\times$ is trivial on $Z_{\hat{G}_\text{SC}} \cap S^0_{\phi, \text{sc}}$.
\end{lemma}

By abuse of notations, the so-obtained character of $Z^{\Gamma_F}_{\hat{G}_\text{SC}}/(Z_{\hat{G}_\text{SC}} \cap S^0_{\phi, \text{sc}}) \subset \tilde{Z}_\phi$ is still denoted by $\chi_G$. Also note that $\chi_G$ depends only on $G_\text{AD}$.

\begin{proof}
  Let us reproduce the proof in \cite{HS12} here. Let $M = M_\phi$ be a minimal Levi subgroup of $G$ through which $\phi$ factorizes (we used the relevance condition here). By the recollections in \S\ref{sec:res-L}, $Z^{\Gamma_F, 0}_{\hat{M}_\text{sc}}$ is a maximal torus in $S^0_{\phi, \text{sc}}$. Therefore
  $$ S^0_{\phi, \text{sc}} \cap Z_{\hat{G}_\text{SC}} = Z^{\Gamma_F, 0}_{\hat{M}_\text{sc}} \cap Z^{\Gamma_F}_{\hat{G}_\text{SC}}, $$
  and the last group is contained in $\Ker(\chi_G)$ by \cite[Corollary 2.2]{Ar99}.
\end{proof}

Consider the familiar situation $G_\text{der} \subset G^\sharp \subset G$ (cf. \S\ref{sec:res-L}), so that we have the $\Gamma_F$-equivariant central extension
$$ 1 \to \hat{Z}^\sharp \to \hat{G} \xrightarrow{\mathbf{pr}} \hat{G^\sharp} \to 1. $$

Let $\phi \in \Phi(G)$ and $\phi^\sharp := \mathbf{pr} \circ \phi \in \Phi(G^\sharp)$. The definitions above pertain to $(G^\sharp, \phi^\sharp)$ as well. Set
$$ X^G(\phi) := \left\{ \mathbf{a} \in H^1_\text{cont}(W_F, \hat{Z}^\sharp) : \mathbf{a}\phi \sim \phi \right\}. $$

This is a finite abelian group (cf. the proof of Theorem \ref{prop:lifting-parameter}).

\begin{lemma}\label{prop:S-sequence}
  Let $s \in S_{\phi^\sharp, \mathrm{ad}}$, regarded as an element of $\widehat{G^\sharp}/Z^{\Gamma_F}_{\widehat{G^\sharp}}$. Then $s$ determines a class $\mathbf{a} \in H^1_\mathrm{cont}(W_F, \hat{Z}^\sharp)$ characterized by
  \begin{gather}\label{eqn:s-a-centralize}
    \tilde{s} \phi(w) \tilde{s}^{-1} = a(w) \phi(w), \quad w \in \WD_F
  \end{gather}
  where
  \begin{itemize}
    \item $\tilde{s} \in \hat{G}$ is a lifting of $s$;
    \item $a: W_F \to \hat{Z}^\sharp$ is some $1$-cocycle representing $\mathbf{a}$, inflated to $\WD_F$.
  \end{itemize}
  This induces an exact sequence
  $$ \mathscr{S}_\phi \to \mathscr{S}_{\phi^\sharp} \to X^G(\phi) \to 1. $$
\end{lemma}
\begin{proof}
  Choose a lifting $\tilde{s} \in \hat{G}$. Since $s$ centralizes $\phi^\sharp$, there exists a continuous function $a: \WD_F \to \hat{Z}^\sharp$ satisfying \eqref{eqn:s-a-centralize}. It is straightforward to check that $a$ is inflated from a $1$-cocycle $W_F \to \hat{Z}^\sharp$. The $1$-cocycle $a$ does depend on the choice of $\tilde{s}$, but its class $\mathbf{a} \in H^1_\text{cont}(W_F, \hat{Z}^\sharp)$ is uniquely determined by $s$; it is also obvious that $s \mapsto \mathbf{a}$ is a homomorphism. Conversely, every $s$ that satisfies \eqref{eqn:s-a-centralize} for some $\tilde{s}, a$ clearly belongs to $S_{\phi^\sharp, \text{ad}}$. Hence the image of $s \mapsto \mathbf{a}$ equals $X^G(\phi)$, by the very definition of $X^G(\phi)$.

  If $s$ is mapped to the trivial class in $X^G(\phi)$, then we may choose $\tilde{s}$ so that $\tilde{s} \phi \tilde{s}^{-1} = \phi$; therefore $s$ comes from $S_{\phi, \text{ad}}$, and vice versa. Hence we have an exact sequence of locally compact groups
  $$ S_{\phi, \mathrm{ad}} \to S_{\phi^\sharp, \mathrm{ad}} \to X^G(\phi) \to 1 . $$

  By a connectedness argument, we may pass from the $S$-groups to the $\mathscr{S}$-groups that gives the asserted exact sequence.
\end{proof}

\paragraph{The case of the inner forms of $\SL(N)$}
Let us revert to the situation where
\begin{align*}
  G & = \GL_D(n), \\
  G^* & = \GL_F(N), \\
  G^\sharp & = \SL_D(n), \\
  \chi_G: & Z_{\hat{G}_\text{SC}} = \mu_N(\C) \to \C^\times.
\end{align*}

It is well-known that $\mathscr{S}_\phi = \{1\}$ for every $\phi \in \Phi(G^*)$. Indeed, $Z_{\hat{G}}(\Im(\phi))$ is a principal Zariski open subset in some linear subspace of $\text{Mat}_{N \times N}(\C)$, thus is connected; so is its quotient by $Z^{\Gamma_F}_{\hat{G}} = Z_{\hat{G}} = \C^\times$.

Let $\phi^\sharp \in \Phi(G^\sharp)$ with a lifting $\phi \in \Phi(G)$. Hence Lemma \ref{prop:S-sequence} yields a canonical isomorphism
$$ \mathscr{S}_{\phi^\sharp} \rightiso X^G(\phi). $$

Assume henceforth that $\phi^\sharp \in \Phi_{G^*-\text{gen}}(G^\sharp)$, so $\phi \in \Phi_{G^*-\text{gen}}(G)$ as well. The local Langlands correspondence for $G$ (Theorem \ref{prop:LLC-G}) is thus applicable. Since the local Langlands correspondence is compatible with twisting by characters, we have $X^G(\phi) = X^G(\pi)$ where $\pi \leftrightarrow \phi$. Therefore we deduce the natural isomorphism
\begin{gather}\label{eqn:S-X-isom}
  \mathscr{S}_{\phi^\sharp} \rightiso X^G(\pi), \quad \text{where } \pi \leftrightarrow \phi, \; \phi^\sharp = \mathbf{pr} \circ \phi .
\end{gather}

Also observe that $\chi_G$ induces a character of $\tilde{Z}_{\phi^\sharp}$ by Lemma \ref{prop:chi-G-factorization}, since $\Gamma_F$ acts trivially on $Z_{\hat{G}_\text{SC}}$.

\begin{theorem}[{\cite[Lemma 12.5]{HS12}}]\label{prop:Lambda}
  Let $\phi^\sharp \in \Phi_{G^*-\mathrm{gen}}(G^\sharp)$ with a chosen lifting $\phi \in \Phi_{G^*-\mathrm{gen}}(G)$. Let $\pi \in \Pi_{G^*-\mathrm{gen}}(G)$ such that $\pi \leftrightarrow \phi$ by Theorem \ref{prop:LLC-G}. Then there exists a homomorphism
  $$ \Lambda: \tilde{\mathscr{S}}_{\phi^\sharp} \to S^G(\pi) $$
  such that the following diagram is commutative with exact rows:
  $$\xymatrix{
    1 \ar[r] & \tilde{Z}_{\phi^\sharp} \ar[d]_{\chi_G} \ar[r] & \tilde{\mathscr{S}}_{\phi^\sharp} \ar[d]^{\Lambda} \ar[r] & \mathscr{S}_{\phi^\sharp} \ar[d]^{\simeq} \ar[r] & 1 \\
    1 \ar[r] & \C^\times \ar[r] & S^G(\pi) \ar[r] & X^G(\pi) \ar[r] & 1
  }$$
  where the rightmost vertical arrow is that of \eqref{eqn:S-X-isom}.

  Moreover, $\Lambda$ is unique up to $\Hom(X^G(\pi),\C^\times)$, i.e. up to the automorphisms of the lower central extension, and upon identifying $\mathscr{S}_{\phi^\sharp}$ and $X^G(\pi)$, this diagram is a push-forward of central extensions by $\chi_G$.
\end{theorem}

Note that the assertions about uniqueness and the push-forward are evident; the upshot is the existence of $\Lambda$.

Let $\phi^\sharp$, $\phi$, $\pi$ be as above. Put
$$ \Pi(\tilde{\mathscr{S}}_{\phi^\sharp}, \chi_G) := \left\{ \rho \in \Pi(\tilde{\mathscr{S}}_{\phi^\sharp}) : \forall z \in \tilde{Z}_{\phi^\sharp}, \; \rho(z) = \chi_G(z) \identity \right\}. $$

The homomorphism $\Lambda$ in Theorem \ref{prop:Lambda} induces a bijection $\Pi(\tilde{\mathscr{S}}_{\phi^\sharp}, \chi_G) \rightiso \Pi_-(S^G(\pi))$. Recall that in the local Langlands correspondence for $G^\sharp$ (Theorem \ref{prop:LLC-SL}), the packet $\Pi_{\phi^\sharp}$ attached to $\phi^\sharp$ is defined as $\Pi_\pi$, the set of irreducible constituents of $\pi|_{G^\sharp}$. Combining Theorem \ref{prop:S-decomp} with Theorem \ref{prop:Lambda}, we arrive at the following description of the packet $\Pi_{\phi^\sharp}$.

\begin{corollary}\label{prop:S-S}
  Let $\phi^\sharp$, $\phi$, $\pi$ be as above. Let $\Hom(X^G(\pi), \C^\times)$ act on $\Pi_{\phi^\sharp}$ via the canonical isomorphisms $\Pi_{\phi^\sharp} = \Pi_\pi = \Pi_-(S^G(\pi))$. Then there is a bijection
  $$ \Pi(\tilde{\mathscr{S}}_{\phi^\sharp}, \chi_G) \rightiso \Pi_{\phi^\sharp}, $$
  which is canonical up to the $\Hom(X^G(\pi), \C^\times)$-action on $\Pi_{\phi^\sharp}$.
\end{corollary}

When $G$ is quasisplit, $\chi_G$ will be trivial and $\Pi(\tilde{\mathscr{S}}_{\phi^\sharp}, \chi_G) = \Pi(\mathscr{S}_{\phi^\sharp})$; the bijection in the Corollary can then be normalized by choosing a Whittaker datum for $G^\sharp$ (cf. \cite[Chapter 3]{HS12}). In general, however, there is no reason to expect a canonical choice of the bijection $\Pi(\tilde{\mathscr{S}}_{\phi^\sharp}, \chi_G) \rightiso \Pi_{\phi^\sharp}$.

\subsection{Generalization}\label{sec:generalization}
Consider the following abstract setting.
\begin{itemize}
  \item Let $M$, $M^\sharp$, $M^\sharp_0$ be connected reductive $F$-groups such that $M$ has an split inner form $M^*$, and
    $$ M_\text{der} \subset M^\sharp_0 \subset M^\sharp \subset M . $$
  \item For $\pi \in \Pi(M)$, let $S^M(\pi)$ and $X^M(\pi)$ be the groups defined in \S\ref{sec:res-rep} relative to $M^\sharp$, and denote by $S^M_0(\pi)$, $X^M_0(\pi)$ the groups defined relative to $M^\sharp_0$.
  \item Assume that there are subsets $\Pi_{\text{gen}}(M)$ and $\Phi_{\text{gen}}(M)$ of $\Pi(M)$ and $\Phi(M)$, respectively, together with a ``local Langlands correspondence'' $\pi \leftrightarrow \phi$ between $\Pi_{\text{gen}}(M)$ and $\Phi_{\text{gen}}(M)$ that is compatible with twist by characters, as in Theorem \ref{prop:LLC-G}. We define $\Phi_{\text{gen}}(M^\sharp)$ (resp. $\Phi_{\text{gen}}(M^\sharp_0)$) to be the set of L-parameters that lift to $\Pi_{\text{gen}}(M^\sharp)$ (resp. $\Pi_{\text{gen}}(M^\sharp_0)$) via Theorem \ref{prop:lifting-parameter}.
  \item Assume that $\mathscr{S}_\phi = \{1\}$ for every $\phi \in \Phi_{\text{gen}}(M)$.
\end{itemize}

Let $\phi \in \Phi_{\text{gen}}(M)$ and $\pi \in \Pi_\text{gen}(M)$ such that $\pi \leftrightarrow \phi$. As before, we deduce L-parameters $\phi^\sharp \in \Phi_{\text{gen}}(M^\sharp)$ and $\phi^\sharp_0 \in \Phi_{\text{gen}}(M^\sharp_0)$. First of all, let $\bullet$ be one of the subscripts ``ad'' or ``sc''. We have $S_{\phi^\sharp, \bullet} \subset S_{\phi^\sharp_0, \bullet}$ and $S^0_{\phi^\sharp, \bullet} \subset S^0_{\phi^\sharp_0, \bullet}$. In view of the definitions in \S\ref{sec:iden-S-groups}, we deduce natural isomorphisms $\mu$, $\tilde{\mu}$ that fit into the following commutative diagram.
\begin{gather}\label{eqn:mu}
\xymatrix{
  \mathscr{S}_{\phi^\sharp} \ar[rr]^{\mu} & & \mathscr{S}_{\phi^\sharp_0} \\
  \tilde{\mathscr{S}}_{\phi^\sharp} \ar[u] \ar[rr]^{\tilde{\mu}} & & \tilde{\mathscr{S}}_{\phi^\sharp_0} \ar[u] \\
  \tilde{Z}_{\phi^\sharp} \ar[u] \ar@{->>}[rr]^{\tilde{\mu}} \ar[rd]_{\chi_M} & & \tilde{Z}_{\phi^\sharp_0} \ar[u] \ar[ld]^{\chi_M} \\
  & \C^\times &
}\end{gather}

Secondly, we have $X^M(\pi) \subset X^M_0(\pi)$ as subgroups of $M(F)^D$. Consequently $S^M(\pi) \subset S^M_0(\pi)$. Iterating the arguments for \eqref{eqn:S-X-isom}, we obtain the commutative diagram
\begin{gather}\label{eqn:mu-2}
\xymatrix{
  \mathscr{S}_{\phi^\sharp} \ar[r]^{\simeq} \ar@{^{(}->}[d]_{\mu} & X^M(\pi) \ar@{^{(}->}[d] \\
  \mathscr{S}_{\phi^\sharp_0} \ar[r]^{\simeq} & X^M_0(\pi).
}\end{gather}

\begin{theorem}\label{prop:generalization}
  Let $\phi, \phi^\sharp, \phi^\sharp_0$ and $\pi$ be as above. Assume that there exists a homomorphism $\Lambda_0: \tilde{\mathscr{S}}_{\phi^\sharp_0} \to S^M_0(\pi)$ such that the following diagram is commutative with exact rows:
  \begin{gather}\label{eqn:comm-diag-0}
  \xymatrix{
    1 \ar[r] & \tilde{Z}_{\phi^\sharp_0} \ar[d]_{\chi_M} \ar[r] & \tilde{\mathscr{S}}_{\phi^\sharp_0} \ar[d]^{\Lambda_0} \ar[r] & \mathscr{S}_{\phi^\sharp_0} \ar[d]^{\simeq} \ar[r] & 1 \\
    1 \ar[r] & \C^\times \ar[r] & S^M_0(\pi) \ar[r] & X^M_0(\pi) \ar[r] & 1
  }\end{gather}
  Then by setting $\Lambda := \Lambda_0 \circ \tilde{\mu} : \tilde{\mathscr{S}}_{\phi^\sharp} \to S^M_0(\pi)$ (cf. \eqref{eqn:mu}), the image of $\Lambda$ lies in $S^M(\pi)$ and the analogous diagram below is commutative
  \begin{gather}\label{eqn:comm-diag-1}
  \xymatrix{
    1 \ar[r] & \tilde{Z}_{\phi^\sharp} \ar[d]_{\chi_M} \ar[r] & \tilde{\mathscr{S}}_{\phi^\sharp} \ar[d]^{\Lambda} \ar[r] & \mathscr{S}_{\phi^\sharp} \ar[d]^{\simeq} \ar[r] & 1 \\
    1 \ar[r] & \C^\times \ar[r] & S^M(\pi) \ar[r] & X^M(\pi) \ar[r] & 1 .
  }\end{gather}

  Consequently, there is a bijection
  $$ \Pi(\tilde{\mathscr{S}}_{\phi^\sharp}, \chi_G) \rightiso \Pi_{\phi^\sharp}, $$
  which is canonical up to the $\Hom(X^G(\pi), \C^\times)$-action on $\Pi_{\phi^\sharp}$.
\end{theorem}
\begin{proof}
  Let $\tilde{s} \in \tilde{\mathscr{S}}_{\phi^\sharp}$, denote by $s$ its image in $\mathscr{S}_{\phi^\sharp}$ and set $s_0 := \mu(s) \in \mathscr{S}_{\phi^\sharp_0}$. Let $\eta \in X^M_0(\pi)$ be the character coming from $\Lambda(\tilde{s}) := \Lambda_0(\tilde{\mu}(\tilde{s})) \in S^M_0(\pi)$. Then by \eqref{eqn:comm-diag-0} and \eqref{eqn:mu}, $\eta$ is the image of $s_0$ under $\mathscr{S}_{\phi^\sharp_0} \rightiso X^M_0(\pi)$; using \eqref{eqn:mu-2}, it is also the image of $s$ under $\mathscr{S}_{\phi^\sharp} \rightiso X^M(\pi)$. If we can show $\Lambda(\tilde{s}) \in S^M(\pi)$ for all $\tilde{s}$, then the rightmost square in \eqref{eqn:comm-diag-1} will commute. Since the square
  $$\xymatrix{
    S^M(\pi) \ar@{^{(}->}[d] \ar@{->>}[r] & X^M(\pi) \ar@{^{(}->}[d] \\
    S^M_0(\pi) \ar@{->>}[r] & X^M_0(\pi) \\
  }$$
  is commutative and cartesian for trivial reasons, it follows that $\Lambda(\tilde{s}) \in S^M(\pi)$. Hence the image of  $\Lambda$ lies in $S^M(\pi)$. This also finishes the commutativity of the rightmost square in \eqref{eqn:comm-diag-1}.
  
  Consider the leftmost square in \eqref{eqn:comm-diag-1}. It follows from \eqref{eqn:mu} that for all $z \in \tilde{Z}_{\phi^\sharp}$, we have
  $$ \Lambda(z) = \Lambda_0(\tilde{\mu}(z)) = \chi_M(\tilde{\mu}(z)) = \chi_M(z). $$
  Hence the leftmost square is commutative as well.

  The bijection $\Pi(\tilde{\mathscr{S}}_{\phi^\sharp}, \chi_G) \rightiso \Pi_{\phi^\sharp}$ follows easily from the previous assertions, as in the proof of Corollary \ref{prop:S-S}.
\end{proof}

\begin{remark}
  The conditions of the Theorem are satisfied if $M$ is a Levi subgroup of $\GL_D(n)$, say of the form
  $$ M = \prod_{i \in I} \GL_D(n_i), \quad \sum_{i \in I} n_i = n $$
  and
  $$ M^\sharp_0 := M_\text{der} = \prod_{i \in I} \SL_D(n_i). $$
  We simply set $\Pi_{\text{gen}}(M) := \Pi_{M^*-\text{gen}}(M)$ and $\Phi_{\text{gen}}(M) := \Phi_{M^*-\text{gen}}(M)$ by a straightforward generalization of the definitions in \S\ref{sec:LLC}. The correspondence $\pi \leftrightarrow \phi$ follows from that in Theorem \ref{prop:LLC-G}, applied to each index $i \in I$. The group $M^\sharp$ can be any intermediate group between $M^\sharp_0$ and $M$, including the important case that
  $$ M^\sharp := M \cap \SL_D(n) = \left\{ (x_i)_{i \in I} \in M : \prod_{i \in I} \Nrd(x_i) = 1 \right\} .$$
  Therefore, Theorem \ref{prop:Lambda} and Corollary \ref{prop:S-S} can be generalized to the Levi subgroups of $\SL_D(n)$.

  Indeed, it suffices to verify the commutativity of the diagram \eqref{eqn:comm-diag-0}. Writing $\pi = \boxtimes_{i \in I} \pi_i$ and $\phi = (\phi_i)_{i \in I}$, the results in \S\ref{sec:iden-S-groups} applied to each $i \in I$ gives a commutative diagram similar to \eqref{eqn:comm-diag-0}, except that its bottom row is the central extension
  \begin{gather}\label{eqn:another-bottom-row}
    1 \to (\C^\times)^I \to \prod_{i \in I} S^{\GL_D(n_i)}(\pi_i) \to X^G_0(\pi) \to 1,
  \end{gather}
  and that $\chi_M$ is replaced by
  $$ \prod_{i \in I} \chi_{\GL_D(n_i)}: \tilde{Z}_{\phi^\sharp_0} \to (\C^\times)^I .$$
  To obtain the desired short exact sequence, it remains to take the push-forward of \eqref{eqn:another-bottom-row} by the multiplication map $(\C^\times)^I \to \C^\times$.
\end{remark}


\section{The dual $R$-groups}\label{sec:dual-R}
\subsection{A commutative diagram}
As in \S\ref{sec:groups}, we take 
\begin{align*}
  G & = \GL_D(n), \\
  G^* & = \GL_F(N), \\
  G^\sharp & = \SL_D(n)
\end{align*}
We also fix a Levi subgroup $M$ of $G$ and set $M^\sharp := M \cap G^\sharp$.

To define the dual groups $\Lgrp{G}$, $\Lgrp{M}$, etc., we fix a quasisplit inner twist $\psi: G \times_F \bar{F} \rightiso G^* \times_F \bar{F}$ which restricts to a quasisplit inner twist $M \times_F \bar{F} \rightiso M^* \times_F \bar{F}$, as well as an $F$-splitting for $G^*$ that is compatible with $M^*$. Therefore there is a canonical L-embedding $\Lgrp{M} \hookrightarrow \Lgrp{G}$. Idem for $\Lgrp{G^\sharp}$ and $\Lgrp{M^\sharp}$.

As usual, the natural projections $\Lgrp{G} \to \Lgrp{G^\sharp}$ and $\Lgrp{M} \to \Lgrp{M^\sharp}$ are denoted by $\mathbf{pr}$. Put
\begin{align*}
  A_{\widehat{M^\sharp}} & := Z_{\widehat{M^\sharp}} = Z^{\Gamma_F, 0}_{\widehat{M^\sharp}} \hookrightarrow \widehat{G^\sharp}.
\end{align*}

Consider $\phi_M \in \Phi_{2,\text{bdd}}(M)$. Let $\phi$ be its composition with $\Lgrp{M} \hookrightarrow \Lgrp{G}$. Set $\phi^\sharp_M := \mathbf{pr} \circ \phi_M \in \Phi_{2,\text{bdd}}(M^\sharp)$ and $\phi^\sharp := \mathbf{pr} \circ \phi \in \Phi_{\text{bdd}}(G^\sharp)$. Every $\phi^\sharp \in \Phi_{\text{bdd}}(G^\sharp)$ is obtained in this way (recall Theorem \ref{prop:lifting-parameter}).

The construction of the dual $R$-group associated to $\phi^\sharp$, denoted by $R_{\phi^\sharp}$, is given as follows. Define
\begin{align*}
  N_{\phi^\sharp, \text{ad}} & := N_{S_{\phi^\sharp, \text{ad}}}(A_{\widehat{M^\sharp}}), \\
  \mathfrak{N}_{\phi^\sharp} & := \pi_0(N_{\phi^\sharp, \text{ad}}, 1), \\
  W_{\phi^\sharp} & := W(S_{\phi^\sharp, \text{ad}}, A_{\widehat{M^\sharp}}) \hookrightarrow W^{\hat{G}}(\hat{M}), \\
  W^0_{\phi^\sharp} & := W(S^0_{\phi^\sharp, \text{ad}}, A_{\widehat{M^\sharp}}) \; \triangleleft W_{\phi^\sharp} , \\
  R_{\phi^\sharp} & := W_{\phi^\sharp}/W^0_{\phi^\sharp}.
\end{align*}
The meaning of $W(\cdots, \cdots)$ is as follows: for any pair of complex groups $a \subset A$, the symbol $W(A,a)$ denotes the group $N_A(a)/Z_A(a)$. Note that $W^0_{\phi^\sharp}$ is the Weyl group associated to some root system, as $S^0_{\phi^\sharp, \text{ad}}$ is connected and reductive.

Since the centralizer of $A_{\widehat{M^\sharp}}$ in the connected reductive group $S^0_{\phi^\sharp, \text{ad}}$ is connected, there exists a canonical injection $W^0_{\phi^\sharp} \hookrightarrow \mathfrak{N}_{\phi^\sharp}$. From the results recalled in \S\ref{sec:res-L}, the torus $A_{\widehat{M^\sharp}}$ is a maximal torus in $S^0_{\phi^\sharp, \text{ad}}$. Using the conjugacy of maximal tori, one sees that the inclusion map $N_{\phi^\sharp, \text{ad}} \hookrightarrow S_{\phi^\sharp, \text{ad}}$ induces a canonical isomorphism $\mathfrak{N}_{\phi^\sharp}/W^0_{\phi^\sharp} \rightiso \mathscr{S}_{\phi^\sharp}$.

On the other hand, we also have canonical injections
\begin{gather*}
  \mathscr{S}_{\phi^\sharp_M} \hookrightarrow \mathscr{S}_{\phi^\sharp}, \\
  \mathscr{S}_{\phi^\sharp_M} \hookrightarrow \mathfrak{N}_{\phi^\sharp}.
\end{gather*}
The first one follows from the fact that $Z^{\Gamma_F}_{\widehat{M^\sharp}} = Z^{\Gamma_F}_{\widehat{G^\sharp}} Z^{\Gamma_F, 0}_{\widehat{M^\sharp}}$ (see \cite[Lemma 1.1]{Ar99}). The injectivity of the second map follows; moreover, its image is characterized as the elements fixing $A_{\widehat{M^\sharp}}$ pointwise.

The relations among these groups are recapitulated in the following result.

\begin{proposition}[{\cite[\S 7]{Ar89-unip}}]\label{prop:diagram}
  The groups above fit into a commutative diagram
  $$\xymatrix{
    & & 1 \ar[d] & \ar[d] 1 & \\
    & & W^0_{\phi^\sharp} \ar@{=}[r] \ar[d] & W^0_{\phi^\sharp} \ar[d] & \\
    1 \ar[r] & \mathscr{S}_{\phi^\sharp_M} \ar[r] \ar@{=}[d] & \mathfrak{N}_{\phi^\sharp} \ar[r] \ar[d] & W_{\phi^\sharp} \ar[r] \ar[d] & 1 \\
    1 \ar[r] & \mathscr{S}_{\phi^\sharp_M} \ar[r] & \mathscr{S}_{\phi^\sharp} \ar[r] \ar[d] & R_{\phi^\sharp} \ar[r] \ar[d] & 1 \\
    & & 1 & 1 & \\
  }$$
  whose rows and columns are exact.
\end{proposition}
The arrow $\mathscr{S}_{\phi^\sharp} \to R_{\phi^\sharp}$ is uniquely determined by the other terms in this diagram. Cf. the proof of Lemma \ref{prop:X-W} below.

The same constructions can be applied to $\phi$ and $\phi_M$. The corresponding objects are denoted by $W_{\phi}$, $W^0_\phi$, etc.

Upon identifying $W^{\hat{G}}(\hat{M})$ and $W^G(M)$, we can make $W_{\phi^\sharp}$ act on the tempered L-packet $\Pi_{\phi^\sharp}$. For any $\sigma^\sharp \in \Pi_{\phi^\sharp}$, define
\begin{align*}
  W_{\phi^\sharp, \sigma^\sharp} & := \text{Stab}_{W_{\phi^\sharp}}(\sigma^\sharp), \\
  W^0_{\phi^\sharp, \sigma^\sharp} & := \text{Stab}_{W^0_{\phi^\sharp}}(\sigma^\sharp), \\
  R_{\phi^\sharp, \sigma^\sharp} & := W_{\phi^\sharp, \sigma^\sharp}/W^0_{\phi^\sharp, \sigma^\sharp}.
\end{align*}

The last object $R_{\phi^\sharp, \sigma^\sharp}$, viewed as a subgroup of $R_{\phi^\sharp}$, is what we want to compare with the Knapp-Stein $R$-group $R_{\sigma^\sharp}$.

\subsection{Identification of $R$-groups}\label{sec:iden-R-groups}
Retain the notations of the previous subsection and fix a parabolic subgroup $P \in \mathcal{P}(M)$. We shall always make the identification $W^G(M)=W^{G^\sharp}(M^\sharp)$. The results in \S\ref{sec:res-2} are applicable to tempered representations of $M(F)$ by Theorem \ref{prop:irred}.

Henceforth, let $\sigma \in \Pi_{2,\text{temp}}(M)$ (resp. $\pi \in \Pi_{\text{temp}}(G)$) be the representations corresponding to $\phi_M$ (resp. $\phi$) by Theorem \ref{prop:LLC-G}. Then we have
$$ \pi \simeq I^G_P(\sigma). $$

Recall that we have defined canonical isomorphisms $\mathscr{S}_{\phi^\sharp} \rightiso X^G(\pi)$ and $\mathscr{S}_{\phi^\sharp_M} \rightiso X^M(\sigma)$ in \S\S\ref{sec:iden-S-groups}-\ref{sec:generalization}. By inspecting the construction in Lemma \ref{prop:S-sequence}, we see that these two isomorphisms are compatible with the embeddings $\mathscr{S}_{\phi^\sharp_M} \hookrightarrow \mathscr{S}_{\phi}$ and $X^M(\sigma) \hookrightarrow X^G(\sigma)$. Therefore we obtain $\gamma: X^G(\pi)/X^M(\sigma) \rightiso \mathscr{S}_{\phi^\sharp}/\mathscr{S}_{\phi^\sharp_M}$.

\begin{lemma}\label{prop:X-W}
  Define an isomorphism
  $$ {\hat{\Gamma}}^{-1}: X^G(\pi)/X^M(\sigma) \xrightarrow{\gamma} \mathscr{S}_{\phi^\sharp}/\mathscr{S}_{\phi^\sharp_M} \rightiso W_{\phi^\sharp}/W^0_{\phi^\sharp} =: R_{\phi^\sharp}, $$
  where the second arrow is given by Proposition \ref{prop:diagram}. Then it is characterized by the equation
  $$ \eta\sigma \simeq w\sigma $$
  whenever
  $$ {\hat{\Gamma}}^{-1}(\eta \text{ mod } X^M(\sigma)) = w \text{ mod } W^0_{\phi^\sharp}, $$
  for all $\eta \in X^G(\pi)$ and $w \in W_{\phi^\sharp}$.
\end{lemma}
As what the notation suggests, we set $\hat{\Gamma}$ to be the inverse of ${\hat{\Gamma}}^{-1}$.

\begin{proof}
  The equation $\eta\sigma \simeq w\sigma$ clearly characterizes ${\hat{\Gamma}}^{-1}$. Let $\eta \in X^G(\pi)$ and $\mathbf{a} \in H^1_{\text{cont}}(W_F, \hat{Z}^\sharp)$ which corresponds to $\eta$, together with a chosen $1$-cocycle $a$ in the cohomology class of $\mathbf{a}$. Since $L(\sigma)=X^G(\pi)$ by Proposition \ref{prop:L-X^G}, there exists $w \in W^G(M)$ such that $\eta\sigma \simeq w\sigma$. On the dual side, it implies that there exists $t \in N_{\hat{G}}(\hat{M})$ representing $w$, such that
  $$ t\phi t^{-1} = a\phi. $$

  This implies that $t \text{ mod } Z_{\hat{G}}$ belongs to $S_{\phi^\sharp, \text{ad}}$, and its class $[t]$ in $\mathscr{S}_{\phi^\sharp}$ corresponds to $\eta$ (recall the construction in Lemma \ref{prop:S-sequence}).

  On the other hand, we have $t \in N_{\phi^\sharp, \text{ad}}$ and its class $[\mathfrak{t}]$ in $\mathfrak{N}_{\phi^\sharp}$ is mapped to $[t]$ under the arrow $\mathfrak{N}_{\phi^\sharp} \twoheadrightarrow \mathscr{S}_{\phi^\sharp}$ in Proposition \ref{prop:diagram}. One can also apply the arrow $\mathfrak{N}_{\phi^\sharp} \twoheadrightarrow W_{\phi^\sharp}$ to $[\mathfrak{t}]$; since $W_{\phi^\sharp}$ is identified as a subgroup of $W^G(M)$, the image is simply $w$.

  Upon some contemplation of the diagram in Proposition \ref{prop:diagram}, one can see that the image of $[t]$ under $\mathscr{S}_{\phi^\sharp} \to R_{\phi^\sharp}$ is just $w$ modulo $W^0_{\phi^\sharp}$. This completes the proof.
\end{proof}

Now recall that we have defined the group $\bar{W}_{\sigma} \subset W^G(M)$. In view of Lemma \ref{prop:barGamma} and Proposition \ref{prop:L-X^G}, we have a canonical isomorphism
$$ \Gamma: \bar{W}_\sigma/W_\sigma \rightiso X^G(\pi)/X^M(\sigma). $$
This is to be compared with $\hat{\Gamma}: W_{\phi^\sharp}/W^0_{\phi^\sharp} \rightiso X^G(\pi)/X^M(\sigma)$.

\begin{proposition}\label{prop:R-1}
  We have
  \begin{enumerate}
    \item $\bar{W}_\sigma = W_{\phi^\sharp}$,
    \item $W_\sigma = W^0_{\phi^\sharp}$,
    \item $\Gamma = \hat{\Gamma}$.
  \end{enumerate}
  In particular, $R_{\phi^\sharp} \simeq X^G(\pi)/X^M(\sigma)$.
\end{proposition}
\begin{proof}
  The first assertion follows from the definition of $\bar{W}_\sigma$ and Theorem \ref{prop:lifting-parameter}. Hence $\Gamma$ and $\hat{\Gamma}$ can be regarded as two surjective homomorphisms from $W_{\phi^\sharp}$ onto $X^G(\pi)/X^M(\sigma)$. However, they admit the same characterization (of the form $\eta\sigma \simeq w\sigma$) by Lemma \ref{prop:X-W} and \ref{prop:barGamma}, hence are equal. This proves the remaining two assertions.
\end{proof}

\begin{proposition}\label{prop:R-2}
  For all $\sigma^\sharp \in \Pi_{\phi^\sharp}$, we have
  \begin{enumerate}
    \item $W_{\sigma^\sharp} = W_{\phi^\sharp, \sigma^\sharp}$,
    \item $W^0_{\sigma^\sharp} = W^0_{\phi^\sharp, \sigma^\sharp}$,
    \item the restriction of $\hat{\Gamma}$ to $W_{\phi^\sharp, \sigma^\sharp}/W^0_{\phi^\sharp, \sigma^\sharp}$ induces an isomorphism
    $$ W_{\phi^\sharp, \sigma^\sharp}/W^0_{\phi^\sharp, \sigma^\sharp} \rightiso Z^M(\sigma)^\perp/X^M(\sigma). $$
  \end{enumerate}

  In particular, $R_{\phi^\sharp, \sigma^\sharp} = R_{\sigma^\sharp}$, and the isomorphisms $\Gamma$, $\hat{\Gamma}$ from these $R$-groups onto $Z^M(\sigma)^\perp/X^M(\sigma)$ coincide (recall Proposition \ref{prop:Goldberg} and Corollary \ref{prop:L-sharp}.)
\end{proposition}
Remainder: the group $Z^M(\sigma)^\perp$ above is defined in \eqref{eqn:perp}.

\begin{proof}
  Our proof is based on the previous result. The first assertion follows immediately from the disjointness of tempered $L$-packets. By Lemma \ref{prop:W^0-equality} and the fact that $W_\sigma = W^0_\sigma$, we have
  \begin{align*}
    W^0_{\phi^\sharp, \sigma^\sharp} & = W^0_{\phi^\sharp} \cap W_{\sigma^\sharp} = W_\sigma \cap W_{\phi^\sharp} \\
    & = W^0_{\sigma^\sharp} \cap W_{\sigma^\sharp} = W^0_{\sigma^\sharp}. 
  \end{align*}
  The second assertion follows. The third assertion is then immediate from Proposition \ref{prop:Z-perp}.
\end{proof}

Note that the proof for the isomorphism $\hat{\Gamma}: R_{\phi^\sharp, \sigma^\sharp} \rightiso Z^M(\sigma)^\perp/X^M(\sigma)$ is independent of the Knapp-Stein theory.

The behaviour of the local Langlands correspondence (Theorem \ref{prop:LLC-SL}) for $G^\sharp$ and its Levi subgroups can now be summarized as follows.

\begin{theorem}\label{prop:main}
  Let $G, G^\sharp$ and $P=MU$, $P^\sharp = M^\sharp U$ be as before. Let $\phi^\sharp_M \in \Phi_{\mathrm{bdd}}(M^\sharp)$ and let $\phi^\sharp \in \Phi_{\mathrm{bdd}}(G^\sharp)$ be the composition of $\phi^\sharp_M$ with $\Lgrp{M^\sharp} \hookrightarrow \Lgrp{G^\sharp}$.
  \begin{enumerate}
    \item For every $\rho \in \Pi(\tilde{\mathscr{S}}_{\phi^\sharp_M}, \chi_M)$ parametrizing an irreducible representation $\sigma^\sharp \in \Pi_{\phi^\sharp_M}$, then $I^{G^\sharp}_{P^\sharp}(\sigma^\sharp)$, regarded as a virtual character of $G^\sharp(F)$, corresponds to that of  $\mathrm{Ind}^{\tilde{\mathscr{S}}_{\phi^\sharp}}_{\tilde{\mathscr{S}}_{\phi^\sharp_M}}(\rho)$.
    \item For any $\sigma^\sharp$ as above, $I^{G^\sharp}_{P^\sharp}(\sigma^\sharp)$ is irreducible if and only if $Z^M(\sigma)^\perp = X^M(\sigma)$ for some (equivalently, for any) $\sigma \in \Pi_{\mathrm{temp}}(M)$ such that $\sigma^\sharp \hookrightarrow \sigma|_{M^\sharp}$.
    \item If $\phi^\sharp_M \in \Phi_{2,\mathrm{bdd}}(M^\sharp)$, then we have natural isomorphisms
      \begin{gather*}
        R_{\phi^\sharp} \simeq X^G(\pi)/X^M(\sigma), \\
        R_{\phi^\sharp, \sigma^\sharp} \simeq R_{\sigma^\sharp} \simeq Z^M(\sigma)^\perp/X^M(\sigma)
      \end{gather*}
      where we set $\pi := I^G_P(\sigma) \in \Pi_{\mathrm{temp}}(G)$, for any choice of $\sigma \in \Pi_{2,\mathrm{temp}}(M)$ such that $\sigma^\sharp \hookrightarrow \sigma|_{M^\sharp}$.
    \item For $\phi^\sharp_M$, $\sigma^\sharp$ and $\rho$ as above, the class $\mathbf{c}_{\sigma^\sharp} \in H^2(R_{\sigma^\sharp}, \C^\times)$ of \eqref{eqn:R-cocycle} corresponds to $\mathbf{c}^{-1}_\rho$, where $\mathbf{c}_\rho \in H^2(R_{\phi^\sharp, \sigma^\sharp}, \C^\times)$ is the obstruction for extending $\rho$ to a representation of the preimage in $\tilde{\mathscr{S}}_{\phi^\sharp}$ of $R_{\phi^\sharp, \sigma^\sharp}$ (see Definition \ref{def:obstruction}).
  \end{enumerate}

  If $G^\sharp$ is quasisplit, then $Z^M(\sigma)^\perp = X^G(\pi)$ and $\tilde{R}_{\sigma^\sharp} \to R_{\sigma^\sharp}$ splits.
\end{theorem}

As mentioned in the Introduction, this settles Arthur's conjectures on $R$-groups for $G^\sharp$.

\begin{proof}
  The first part is nothing but a special case of Proposition \ref{prop:K_0-diagram}. The second part resuls then from the proof of Proposition \ref{prop:Z-perp}; the independence of the choice of $\sigma$ is clear. The third part results from Proposition \ref{prop:R-1} and \ref{prop:R-2}. The fourth part is the combination of Proposition \ref{prop:cocycle} and Theorem \ref{prop:generalization}.

  Finally, $S^G(\pi)$ is commutative when $G^\sharp$ is quasisplit, as $\chi_G=1$. Hence we have $Z^M(\sigma)^\perp = X^G(\pi)$ and $\rho$ can always be extended in that case.
\end{proof}

\begin{remark}
  The decomposition of $I^{G^\sharp}_{P^\sharp}(\sigma^\sharp)$ depends on $\phi^\sharp_M$, but not on the element $\sigma^\sharp$. This is not expected to hold for other groups.
\end{remark}

\begin{remark}\label{rem:nontempered}
  We have limited ourselves to the tempered representations. However, if the local Langlands correspondence (Theorem \ref{prop:LLC-SL}) and Theorem \ref{prop:Lambda} can be extended to Arthur parameters $\psi^\sharp: \WD_F \times \SU(2) \to \Lgrp{G^\sharp}$ (see \cite[\S 6]{Ar89-unip}), then our results should be applicable to Arthur packets $\Pi_{\psi^\sharp}$ as well, except the part concerning the Knapp-Stein $R$-groups $R_{\sigma^\sharp}$. Note that the crucial lifting Theorem \ref{prop:lifting-parameter} also holds for Arthur parameters: see \cite[Remarque 8.2]{Lab85}.
\end{remark}

\subsection{Examples}\label{sec:examples}

The next example on $R$-groups will be constructed using Steinberg representations, whose definitions are reviewed below.

\begin{definition}\label{def:Steinberg}
  For this moment, we assume $G$ to be any connected reductive $F$-group. Fix a minimal parabolic subgroup $P_0$ of $G$. The Steinberg representation $\text{St}_G$ of $G$ is the virtual character of $G(F)$ given by
  $$ \text{St}_G := \sum_{\substack{P \supset P_0 \\ P = M U}} (-1)^{\dim \mathfrak{a}^G_M} \; I^G_P(\delta_P^{-\frac{1}{2}} \mathbbm{1}_M), $$
  where the sum ranges over the parabolic subgroups $P$ containing $P_0$ and $\mathbbm{1}_M$ denotes the trivial representation of $M(F)$.
\end{definition}

The basic fact \cite{Ca73} is that $\text{St}_G$ comes from a smooth irreducible representation in $\Pi_{2, \text{temp}}(G)$, which we denote by the same symbol $\text{St}_G$. It is clearly independent of the choice of $P_0$.

\begin{lemma}\label{prop:St-irred}
  For $G$ as in Definition \ref{def:Steinberg} and a subgroup $G^\sharp$ satisfying $G_{\mathrm{der}} \subset G^\sharp \subset G$, we have
  $$ \mathrm{St}_G|_{G^\sharp} \simeq \mathrm{St}_{G^\sharp}.$$
  In particular, the group $X^G(\mathrm{St}_G)$ defined in \S\ref{sec:res-rep} is trivial.
\end{lemma}
\begin{proof}
  Recall the bijection $P \mapsto P^\sharp := P \cap G^\sharp$ between the parabolic subgroups of $G$ and $G^\sharp$. Since $(\mathbbm{1}_L)|_{L^\sharp} = \mathbbm{1}_{L^\sharp}$ for any Levi subgroup $L$ of $G$, the first isomorphism follows by comparing the formulas defining $\text{St}_G$ and $\text{St}_{G^\sharp}$, together with Lemma \ref{prop:res-induction}. Hence the restriction of $\text{St}_G$ to $G^\sharp$ is irreducible. It follows from Theorem \ref{prop:S-decomp} that $X^G(\text{St}_G) = \{1\}$.
\end{proof}

Let us revert to the setting $G = \GL_D(n)$ and $G^\sharp = \SL_D(n)$.

\begin{example}\label{ex:nonsplit}
  We now set out to construct an example in which $\tilde{R}_{\sigma^\sharp} \to R_{\sigma^\sharp}$ does not split for every $\sigma^\sharp \hookrightarrow \sigma|_M$.

  First of all, there exists $\GL_D(m)$, for some choice of $D,m$, and a representation $\tau \in \Pi_{2,\text{temp}}(\GL_D(m))$ such that $S^{\GL_D(m)}(\sigma)$ is non-commutative. Indeed, for $m=1$ and $D$ equal to the quaternion algebra over $F$, Arthur exhibits in \cite[p.215]{Ar06} an L-parameter $\phi_\tau \in \Phi_{2,\text{temp}}(D^\times)$ such that
  \begin{itemize}
    \item $\tilde{\mathscr{S}}_{\phi^\sharp_\tau}$ is isomorphic to the quaternion group of order $8$;
    \item $\tilde{Z}_{\phi_\tau}$ corresponds to $\{\pm 1\}$.
  \end{itemize}
  In fact, $\phi_\tau$ factors through a homomorphism $\text{Gal}(K/F) \to \PGL(2,\C)$, where $K$ is a biquadratic extension of $F$, whose image is generated by the elements
  $$
    \begin{bmatrix} 1 & 0 \\ 0 & -1 \end{bmatrix}, 
    \begin{bmatrix} 0 & 1 \\ 1 & 0 \end{bmatrix} \in \PGL(2, \C).
  $$

  Since $\chi_{D^\times}$ is injective on $\tilde{Z}_{\phi_\tau}$ by \eqref{eqn:chi-G-special}. Theorem \ref{prop:Lambda} entails that $S^{D^\times}(\tau)$ is non-commutative.

  Take $\eta, \omega \in X^{\GL_D(m)}(\tau)$ so that their preimages in $S^{\GL_D(m)}(\tau)$ do not commute. Let $c$ (resp. $d$) be the order of $\eta$ (resp. $\omega$). Put $\text{St} := \text{St}_{\GL_D(m)}$ and take

  \begin{align*}
    M & := \GL_D(m) \times \prod_{\substack{1 \leq i \leq c \\ 1 \leq j \leq d}} \GL_D(m), \\
    G & := \GL_D(m(cd+1)), \\
    \sigma & := \tau \boxtimes \left( \bigboxtimes_{\substack{1 \leq i \leq c \\ 1 \leq j \leq d}} \eta^{i-1} \omega^{j-1} \text{St} \right), \\
    \pi & := I^G_P(\sigma) \quad \text{ for some } P \in \mathcal{P}(M).
  \end{align*}

  Here $X^M(\sigma)$ is defined relatively to $M^\sharp := G^\sharp \cap M$ where $G^\sharp = \SL_D(m(cd+1))$. The presence of $\text{St}$ forces $X^M(\sigma)$ to be trivial, by Lemma \ref{prop:St-irred}. Hence $\sigma^\sharp := \sigma|_{M^\sharp}$ is irreducible, and it is parametrized by the $1$-dimensional character $\chi_M: \tilde{Z}_{\phi^\sharp_M} \to \C^\times$. According to Theorem \ref{prop:main}, the central extension $\tilde{R}_{\sigma^\sharp} \twoheadrightarrow R_{\sigma^\sharp}$ splits if and only if $\rho$ can be extended to $\tilde{\mathscr{S}}_{\phi^\sharp}$. This is the case if and only if $S^G(\pi) \twoheadrightarrow X^G(\pi)$ splits, by Theorem \ref{prop:Lambda}. Hence it suffices to show the non-commutativity of $S^G(\pi)$.

  Put $L := \GL_D(m) \times \GL_D(mcd) \in \mathcal{L}^G(M)$ and set $\nu := I^L_{P \cap L}(\sigma)$. We claim $\eta,\omega \in S^L(\nu)$. Indeed, the $\GL_D(m)$-component of $L$ does not cause any problem. As for the $\GL_D(mcd)$-component, take representatives $\tilde{w}_\eta$, $\tilde{w}_\eta$ in $\SL_D(mcd)$ of the cyclic permutations
  \begin{gather*}
    w_\eta: 1 \to \cdots \to c \to 1, \\
    w_\omega: 1 \to \cdots \to d \to 1
  \end{gather*}
  of the indexes $i$ and $j$, respectively. Then the intertwining operators $J_\eta, J_\omega$ are given by the operators in \eqref{eqn:r_P} using $\tilde{w}_\eta$, $\tilde{w}_\omega$. Furthermore, $J_\eta$ and $J_\omega$ commute with each other; this follows from \eqref{eqn:r_P-property} and the obvious fact that $\tilde{w}_\eta$ and $\tilde{w}_\omega$ can be chosen to commute.

  From our choice of $\eta,\omega$, it follows that the preimages of $\eta,\omega$ in $S^L(\nu)$ do not commute. Since $S^L(\nu) \hookrightarrow S^G(\pi)$ by Proposition \ref{prop:S-embedding}, $S^G(\pi)$ is non-commutative, as required.
\end{example}

\begin{example}\label{ex:proper-incl}
  Now we set out the show that the inclusion $R_{\phi^\sharp, \sigma^\sharp} \subset R_{\phi^\sharp}$ is proper in general. By Theorem \ref{prop:main} and the notations therein, it amounts to show that $Z^M(\sigma)^\perp \subsetneq X^G(\pi)$ in general.

  As in the previous example, we take some $m \geq 1$, a central division $F$-algebra $D$ and $\tau \in \Pi_{2,\text{temp}}(\GL_D(m))$ such that $X^{\GL_D(m)}(\tau)$ contains $\eta, \omega$ with non-commuting preimages in $S^{\GL_D(m)}(\sigma)$. Take another $\tau' \in \Pi_{2,\text{temp}}(\GL_D(m))$ such that $X^{\GL_D(m)}(\tau') = \angles{\eta}$. Denote by $d$ the order of $\omega$. We take
  \begin{align*}
    M & := \GL_D(m) \times \prod_{1 \leq j \leq d} \GL_D(m), \\
    G & := \GL_D(m(d+1)), \\
    \sigma & := \tau \boxtimes \left( \bigboxtimes_{1 \leq j \leq d} \omega^{j-1} \tau' \right), \\
    \pi & := I^G_P(\sigma) \quad \text{ for some } P \in \mathcal{P}(M).
  \end{align*}
  Therefore $X^M(\sigma) = \angles{\eta}$ (defined relative to $M^\sharp = M \cap G^\sharp$ with $G^\sharp := \SL_D(m(d+1))$ as before), and $S^M(\sigma)$ is commutative. In particular $Z^M(\sigma)=X^M(\sigma)=\angles{\eta}$.

  On the other hand, a variant of the arguments in the previous example show that $\omega, \eta \in X^G(\pi)$ with non-commuting preimages in $S^G(\pi)$. Hence $\omega \in X^G(\pi)$ and $\omega \notin Z^M(\sigma)^\perp$, as required.

  Note that such $\tau$, $\tau'$ do exist when $D$ is the quaternion algebra over $F$ and $m=1$; in that case $\eta, \omega$ are identified with quadratic characters of $F^\times$. Indeed, a candidate of $\tau$ is given in the previous example. On the other hand, to construct $\tau'$ for a given $\eta$, we are reduced to construct $\tau'' \in \Pi_{2,\text{temp}}(\GL_F(2))$ with $X^{\GL_F(2)}(\tau'') = \{1, \eta\}$ and then take $\tau$ to be the Jacquet-Langlands transfer of $\tau''$.

  To finish the construction, let $E$ be the quadratic extension of $F$ determined by $\eta$ and $\theta: E^\times \to \C^\times$ be a continuous character. Set $\tau'' := \Ind_{E/F}(\theta)$ (the local automorphic induction, cf. \cite[Theorem 4.6]{JL70}), then $\eta \tau'' = \tau''$. From \cite[pp.738-739]{LL79}, one sees that $\tau''$ is cuspidal and $|X^{\GL_F(2)}(\tau'')|=2$ for general $\theta$, which suffices to conclude.
\end{example}

\bibliographystyle{abbrv}
\bibliography{Restriction}

\begin{flushleft} \small
  Kuok Fai Chao \\
  Institute of Mathematics, \\
  Academy of Mathematics and Systems Science, \\
  Chinese Academy of Sciences, \\
  55, Zhongguancun East Road, 100190 Beijing \\
  People’s Republic of China. \\
  E-mail address: \texttt{kchao@amss.ac.cn}
\end{flushleft}

\begin{flushleft} \small
  Wen-Wei Li \\
  Morningside Center of Mathematics, \\
  Academy of Mathematics and Systems Science, \\
  Chinese Academy of Sciences, \\
  55, Zhongguancun East Road, 100190 Beijing \\
  People's Republic of China. \\
  E-mail address: \texttt{wwli@math.ac.cn}
\end{flushleft}

\end{document}